\documentclass[12pt]{amsart}

\usepackage{amssymb}
\usepackage{amsmath}
\usepackage{amsthm}

\usepackage{color}

\usepackage[pagebackref,hypertexnames=false, colorlinks, citecolor=red, linkcolor=red]{hyperref} 
\usepackage[backrefs]{amsrefs}

\allowdisplaybreaks

\newcommand\R{\mathbb{R}}

\def\ls{\lesssim}

\newtheorem{thm}{Theorem}[section]
\newtheorem{lem}[thm]{Lemma}
\newtheorem{prop}[thm]{Proposition}

\newtheorem*{TheoremC}{Theorem C}

\numberwithin{equation}{section}

\begin{document}

\arraycolsep=1pt

\title[Two weight Inequality for the Poisson Operator on manifold with ends]{The Two Weight Inequality for the Poisson Semigroup
 on Manifold with Ends}

\author[X. Duong]{Xuan Thinh Duong}
\address{Xuan Thinh Duong, Department of Mathematics \& Statistics\\ Macquarie University\\ NSW, 2109, Australia.}
\email{xuan.duong@mq.edu.au}
\thanks{X. Duong's research supported by ARC DP 190100970. M.-Y. Lee's research supported by Ministry of Science and Technology, R.O.C. grant \#MOST 110-2115-M-008-009-MY2. J. Li's research supported by ARC DP 220100285. B. D. Wick's research supported by National Science Foundation DMS grant \#1800057 and ARC DP 190100970.}

\author[M.-Y. Lee]{Ming-Yi Lee}
\address{Ming-Yi Lee, Department of Mathematics\\ National Central University\\ Chung-Li, 320, Taiwan.}
\email{mylee@math.ncu.edu.tw}

\author[J. Li]{Ji Li}
\address{Ji Li, Department of Mathematics \& Statistics\\ Macquarie University\\ NSW, 2109, Australia.}
\email{ji.li@mq.edu.au}

\author[B. D. Wick]{Brett D. Wick}
\address{Brett D. Wick, Department of Mathematics \& Statistics\\ Washington University -- St. Louis\\ One Brookings Drive\\ St. Louis, MO USA 63130-4899}
\email{wick@math.wustl.edu}


\date{\today}

\begin{abstract}
Let $M = \mathbb R^m \sharp \mathcal R^n$ be a non-doubling manifold with two ends $\mathbb R^m$ and  $\mathcal R^n$,   $m > n \ge 3$.
Let $\Delta$ be the Laplace--Beltrami operator which is non-negative self-adjoint on $L^2(M)$. 
 We give testing conditions for  the two weight inequality for the Poisson semigroup $\mathsf{P}_t= e^{-t\sqrt{\Delta}}$ (generated by $\sqrt{\Delta}$) to hold in this setting.  
 In particular, we prove that for a measure $\mu$ on $M_{+}:=M\times (0,\infty)$ and $\sigma$ on $M$:
$$
\|\mathsf{P}_\sigma(f)\|_{L^2(M_{+};\mu)} \lesssim \|f\|_{L^2(M;\sigma)},
$$
with $\mathsf{P}_\sigma(f)(x,t):= \int_M \mathsf{P}_t(x,y)f(y) \,d\sigma(y)$ (with $\mathsf{P}_t(x,y)$ the Poisson kernel of $\mathsf{P}_t$),
if and only if testing conditions hold for the Poisson semigroup and its adjoint.  Further, the norm of the operator is shown to be equivalent to the best constants in these testing conditions.
\end{abstract}

\maketitle

\section{Introduction and Statement of Main Results\label{s1}}

The interest in the two weight problem stems from a range of applications arising in sophisticated arenas of complex function and spectral theory. 
Suppose $1\le p, q\le \infty, v(x)$ and $w(x)$ are nonnegative measurable functions (i.e. weights) on $\mathbb R^n$ and $\mathbb R^m$ respectively, and that $T$ is an operator taking suitable functions on $\mathbb R^n$ into functions on $\mathbb R^m$. 
In his survey article \cite{M}, Muckenhoupt raised the general question of characterizing when the weighted norm inequality,
$$\bigg(\int_{\mathbb R^m} |Tf(x)|^q w(x)dx \bigg)^{1/q}
 \le C \bigg(\int_{\mathbb R^n} |f(x)|^q v(x)dx\bigg)^{1/p},$$
holds for all appropriate $f$.
Sawyer first introduced testing conditions in \cite{s88-0} (which are now frequently referred to as Sawyer-type testing conditions) 
into the two weight setting for the maximal function, and later in \cite{s88} for the fractional and Poisson integral operators. We refer to \cite[Sec. 12]{L} for applications of two weight inequalities.

In this paper we provide a testing condition for the two weight inequality for the Poisson operator on a non-doubling manifold with ends studied by Grigor’yan and Saloff-Coste \cite{GS}. We note that this could provide a useful tool for the study of the two weight inequality for Riesz transforms in this setting, since the related result in the Euclidean setting plays an important role in the solution to the two weight conjecture for the Hilbert transform (due to Lacey--Sawyer--Shen--Uriarte-Tuero \cite{LSSU} and Lacey \cite{L}) and two weight theorem for $\alpha$-fractional singular integrals (see Sawyer--Shen--Uriarte-Tuero \cite{SSU}) and  for Riesz transforms (Lacey--Wick \cite{LW}).  See also some recent related progress \cite{Hy, GPSSU}.

Concerning the structure of the manifold with ends $M$, we refer the reader to \cite{GS, GIS}. 
The manifold $M$ is basically a copy of $\mathbb R^m$ connected  to $\mathbb R^n \times S^{m-n}$ smoothly by a compact set $K$ of diameter $1$ where
$S^{m-n}$ denotes the unit sphere in $\mathbb R^{m-n}$.  For any $x\in M$, define
$
   |x|:=\sup_{z\in K}d(x,z),
$
where $d=d(x,y)$ is the geodesic distance in $M$.
One can see that $|x|$ is separated from zero on $M$ and $|x|\approx 1+d(x,K)$.  For $x\in M$, let
$
   B(x,r):=\{y\in M: d(x,y)<r\}
$
be the geodesic ball with center $x\in M$ and radius $r>0$ and let
$
   V(x,r)=\nu(B(x,r))
$
where $\nu$ is the Riemannian measure on $M$. We observe that the function $V(x,r)$ satisfies:

(a) $V(x,r)\thickapprox r^m$ for all $x\in M$, when $r\leq 1$;

(b) $V(x,r)\thickapprox r^n$ for $B(x,r)\subset \R^n$, when $r> 1$; and

(c) $V(x,r)\thickapprox r^m$ for $x\in \R^n\backslash K$, $r>2|x|$, or $x\in \R^m$, $r>1$.\\
It is not difficult to check that $M$ does not satisfy the doubling condition.  Indeed, consider a
sequence of balls $B(x_k,r_k)\subset \R^n$ such that $r_k = |x_k| > 1$ and $r_k \rightarrow \infty$ as $k \rightarrow \infty$.
Then $V(x_k,r_k)\thickapprox r_k^n$. However, $V(x_k,2r_k)\thickapprox r_k^m$ and the doubling condition fails since $m>n$.  

Let $\Delta$  be the Laplace-Beltrami operator on $M$ and $e^{-t\sqrt{\Delta}}$ the Poisson semigroup generated by $\sqrt{\Delta}$. We denote 
by $\mathsf{P}_t(x,y)$ the  kernel of the Poisson semigroup $\{e^{-t\sqrt{\Delta}}\}_{t>0}$.

The aim of this paper is to provide a necessary and sufficient condition for a two weight inequality for the Poisson semigroup on a non-doubling manifold with ends.  We state the goal more precisely now. Let  $\sigma$ be a weight on $M$ and $\mu$ be a weight on $M_+=M\times (0,\infty)$.
Consider the inequality
\begin{align}\label{two weight}
\|\mathsf{P}_\sigma(f)\|_{L^2(M_{+};\mu)} \leq \mathcal{N}\|f\|_{L^2(M,\sigma)},
\end{align}
where
\begin{align*}
\mathsf{P}_\sigma(f)(x,t):= \int_M \mathsf{P}_t(x,y)f(y) \,d\sigma(y).
\end{align*}
We use $\mathsf{P}^{*}_\mu$ to denote the dual operator of $\mathsf{P}$, defined as follows
\begin{align*}
\left\langle \mathsf{P}_\sigma(f), g \right\rangle_{L^2(M_{+}^2;\mu)} &=\int_{M_{+}} \int_M \mathsf{P}_t(x,y)f(y) \,d\sigma(y) g(x,t)\,d\mu(x,t)\\
&= \int_M   \int_{M_{+}}  \mathsf{P}_t(x,y)g(x,t)\,d\mu(x,t)   f(y) \,d\sigma(y)\\
&= \int_M  \mathsf{P}^{*}_\mu(g)(y)f(y) \,d\sigma(y).
\end{align*}
So in particular,
\begin{align}\label{P*}
\mathsf{P}^{*}_\mu(g)(y):= \int_{M_{+}} \mathsf{P}_t(x,y)g(x,t) \,d\mu(x,t).
\end{align}
We also observe that a simple duality argument provides:
\begin{align}\label{two weight dual}
\left\Vert \mathsf{P}^{*}_\mu(\phi)\right\Vert_{L^2(M;\sigma)} \lesssim \mathcal{N}\left\Vert \phi\right\Vert_{L^2(M_{+};\mu)}.
\end{align}

The main result of this paper is the following two weight inequality for the Poisson operator $\{\mathsf{P}_t\}_{t>0}$. 
\begin{thm}Let $\sigma$ be a measure on $M$ and $\mu$ a measure on $M_{+}$.
The following conditions are equivalent:
\begin{enumerate}
\item[(1)] The two weight inequality \eqref{two weight} holds.  That is 
\begin{align*}
\|\mathsf{P}_\sigma(f)\|_{L^2(M_{+};\mu)} \leq \mathcal{N}\|f\|_{L^2(M,\sigma)}
\end{align*}
for some  non-negative constant $\mathcal N$; 
\item[(2)] The testing conditions below hold uniformly over all cubes $I\subset M$
\begin{align*}
\int_{\hat{3I}} \mathsf{P}_\sigma(1_I)(x,t)^2 \,d\mu(x,t) & \leq \mathcal{F}^2 \sigma(I),\\
\int_{3I} \mathsf{P}^{*}_\mu(t1_{\hat{I}})(y)^2 \,d\sigma(y) & \leq \mathcal{B}^2 \int_{\hat{I}} t^2 \,d\mu(x,t)
\end{align*}
for some  non-negative constants $\mathcal F$, $\mathcal B$.
\end{enumerate}
Moreover, we have that $\mathcal{N} \approx \mathcal{F}+\mathcal{B}$.
Here $1_I$ is the indicator of $I$, $\hat{I}= I\times [ 0, \ell(I) ]$.  
\end{thm}
It is immediate that the testing conditions are necessary and that $\mathcal{F}+\mathcal{B}\lesssim \mathcal{N}$.  The forward condition follows by testing \eqref{two weight} on an indicator function and restricting the region of integration.  The backward condition follows by testing the dual inequality \eqref{two weight dual} on the indicator of a set and then again restricting the integration.  In the remainder of the paper we address how to show that these testing conditions are sufficient to prove \eqref{two weight} and \eqref{two weight dual}.  In the course of the proof it will also be shown that $\mathcal{N}\lesssim \mathcal{F}+\mathcal{B}$.

We remark that in the setting of manifolds with end, the key difficulty is that the Poisson integrals are essentially different from those in the Euclidean setting. Instead of the standard Euclidean Poisson kernel, a sharp estimate on the Poisson kernels on manifolds with ends splits into 9 different cases according to the position of the variables $x$ and $y$. This will of course entail a case analysis in the proof given below.  Moreover, among those 9 cases, there are 3 cases where the dimension for the scaling (the variable $t$) is different from the dimension of the underlying end of the manifold. To overcome this, we implement a refined decomposition of the underlying end of the manifold, to make good use of the extra space variable in the denominator to compensate for difference between the dimension and scaling.

Throughout the paper we use the notation $X\lesssim Y$ to denote that there is an absolute constant $C$ so that $X\leq CY$.  If we write $X\approx Y$, then we mean that $X\lesssim Y$ and $Y\lesssim X$.  And, $:=$ means equal by definition.

\section{Proof of the two weight Inequality for  \texorpdfstring{$\left\{\mathsf{P}_t\right\}_{t>0}$}{the Bessel Poisson Operator}}
\label{s:MainResult}

We first recall properties of the Poisson kernel $\mathsf{P}_t(x,y)$ on a manifold with ends.

\begin{TheoremC}[\cite{BDLW}]
The Poisson kernel  $\mathsf{P}_t(x,y)$ satisfies the following estimates:

\begin{enumerate}
\item For $x,y\in K $,
$$ \mathsf{P}_t(x,y)\approx \frac{1}{t^m}\Big(\frac{t}{t+d(x,y)}\Big)^{m+ 1}+ \frac{1}{t^n}\Big(\frac{t}{t+d(x,y)}\Big)^{n+ 1}; $$

\item For $x\in \mathbb{R}^m\backslash K $, $y\in K$,
$$ \mathsf{P}_t(x,y)\approx \frac{1}{t^m}\Big(\frac{t}{t+d(x,y)}\Big)^{m+ 1} +\frac{1}{t^n|x|^{m-2}}\Big(\frac{t}{t+d(x,y)}\Big)^{n+ 1}; $$

\item For  $x\in \mathbb{R}^n\backslash K $, $y\in K$,
$$ \mathsf{P}_t(x,y)\approx \frac{1}{t^m}\Big(\frac{t}{t+d(x,y)}\Big)^{m+ 1}+ \frac{1}{t^n}\Big(\frac{t}{t+d(x,y)}\Big)^{n+ 1}; $$

\item For  $x\in \mathbb{R}^m\backslash K $, $y\in \mathbb{R}^n\backslash K $,
\begin{align*} 
\mathsf{P}_t(x,y)&\approx \frac{1}{t^m}\Big(\frac{t}{t+d(x,y)}\Big)^{m+ 1}+ \frac{1}{t^n|x|^{m-2}}\Big(\frac{t}{t+d(x,y)}\Big)^{n+1}\\
&\quad+\frac{1}{t^m|y|^{n-2}}\Big(\frac{t}{t+d(x,y)}\Big)^{m+ 1}; 
\end{align*}

\item For  $x,y\in \mathbb{R}^m\backslash K $,
$$ \mathsf{P}_t(x,y)\approx \frac{1}{t^m}\Big(\frac{t}{t+d(x,y)}\Big)^{m+ 1} +\frac{1}{t^n|x|^{m-2}|y|^{m-2}}\Big(\frac{t}{t+|x|+|y|}\Big)^{n+1};$$ 

\item For $x,y\in \mathbb{R}^n\backslash K $,
$$ \mathsf{P}_t(x,y)\approx \frac{1}{t^m}\Big(\frac{t}{t+d(x,y)}\Big)^{m+ 1}+ \frac{1}{t^n}\Big(\frac{t}{t+d(x,y)}\Big)^{n+ 1}. $$
\end{enumerate}
\end{TheoremC}

We now prove that the testing conditions imply the norm inequality for the Poisson operator $\mathsf{P}_t$ on non-doubling manifold with ends.  The general proof strategy is to follow the line of Sawyer's original argument, \cite{s88}, and use some techniques contained in the proof given by Lacey \cite{L}.  Additional techniques developed here adapting to the Poisson kernel upper bound and the non-doubling measure are also exploited.

To begin with, we assume that $\sigma$ is restricted to some large cube $I_0\subset M_+$, and  that $\mu$ is restricted to $3\hat{I}_0$.  There is no loss in assuming that the measures $\sigma$ and $\mu$ are compactly supported since the resulting estimates will not depend upon the support in any way, and we can then pass to the general case through a standard limiting argument.

To prove \eqref{two weight}, by duality, it suffices to prove that
\begin{align}\label{two weight p}
\int_{M} | \mathsf{P}^{*}_\mu(\phi)(x)|^2 \,d\sigma(x) \lesssim \left(\mathcal{F}^2+\mathcal{B}^2\right)  
\int_{M_{+}} | \phi(y,t)|^2 \,d\mu(y,t).
\end{align}

We split the left-hand side of \eqref{two weight p} as follows.
\begin{align}
&\int_{M} | \mathsf{P}^{*}_\mu(\phi)(x)|^2 \,d\sigma(x)\\
&=\int_{\R^m\backslash K} | \mathsf{P}^{*}_\mu(\phi)(x)|^2 \,d\sigma(x)+\int_{\R^n\backslash K} | \mathsf{P}^{*}_\mu(\phi)(x)|^2 \,d\sigma(x)+\int_{K} | \mathsf{P}^{*}_\mu(\phi)(x)|^2 \,d\sigma(x) \nonumber\\
&=:T_1+T_2+T_3.\nonumber
\end{align}
It suffices to then prove that $T_j\lesssim \left(\mathcal{F}^2+\mathcal{B}^2\right)  \int_{M_{+}} | \phi(y,t)|^2 \,d\mu(y,t) $ for $j=1,2,3$ as this would then prove the main result.

\subsection{Estimate for $T_{1}$}\label{Estimate for $T_{1}$}
Note that in this case, 
from the definition of $P^*$ in \eqref{P*}, we have that for $x\in \R^m\backslash K$,
\begin{align*}
 &\mathsf{P}^{*}_\mu(\phi)(x)\\
 &=  \int_{M_{+}} \mathsf{P}_t(y,x)\phi(y,t) \,d\mu(y,t)\\
 &\le  \int_{M_{+}} \bigg(\frac{1}{t^m}\Big(\frac{t}{t+d(y,x)}\Big)^{m+ 1} +\frac{1}{t^n|x|^{m-2}}\Big(\frac{t}{t+d(x,y)}\Big)^{n+ 1} \\
 &\qquad +\frac{1}{t^m|y|^{n-2}}\Big(\frac{t}{t+d(x,y)}\Big)^{m+ 1}+\frac{1}{t^n|y|^{m-2}|x|^{m-2}}\Big(\frac{t}{t+|y|+|x|}\Big)^{n+1}\bigg)\phi(y,t) \,d\mu(y,t)\\
 &=:\mathsf{P}^{*}_{\mu,1,1}(\phi)(x)+\mathsf{P}^{*}_{\mu,1,2}(\phi)(x)+\mathsf{P}^{*}_{\mu,1,3}(\phi)(x)+\mathsf{P}^{*}_{\mu,1,4}(\phi)(x),
\end{align*}
where $\mathsf{P}^{*}_{\mu,1,1}$ is the operator associated to the integral kernel
$$ \mathsf{P}_{t,1,1}(y,x)=\frac{1}{t^m}\Big(\frac{t}{t+d(y,x)}\Big)^{m+ 1};  $$ 
$\mathsf{P}^{*}_{\mu,1,2}$ is the operator associated to the integral kernel
$$  \mathsf{P}_{t,1,2}(y,x)=\frac{1}{t^n|x|^{m-2}}\Big(\frac{t}{t+d(x,y)}\Big)^{n+ 1};$$
 $\mathsf{P}^{*}_{\mu,1,3}$ is the operator associated to the integral kernel
$$ \mathsf{P}_{t,1,3}(y,x)=\frac{1}{t^m|y|^{n-2}}\Big(\frac{t}{t+d(x,y)}\Big)^{m+ 1};  $$ 
and $\mathsf{P}^{*}_{\mu,1,4}$ is the operator associated to the integral kernel
$$  \mathsf{P}_{t,1,4}(y,x)=\frac{1}{t^n|y|^{m-2}|x|^{m-2}}\Big(\frac{t}{t+|y|+|x|}\Big)^{n+1}.$$
From this decomposition we have
\begin{align*}
T_1&\ls \int_{\R^m\backslash K} | \mathsf{P}^{*}_{\mu,1,1}(\phi)(x)|^2 \,d\sigma(x)
+\int_{\R^m\backslash K} | \mathsf{P}^{*}_{\mu,1,2}(\phi)(x)|^2 \,d\sigma(x)\\
&\qquad+  \int_{\R^m\backslash K} | \mathsf{P}^{*}_{\mu,1,3}(\phi)(x)|^2 \,d\sigma(x)
+\int_{\R^m\backslash K} | \mathsf{P}^{*}_{\mu,1,4}(\phi)(x)|^2 \,d\sigma(x)\\
&=:T_{1,1}+T_{1,2}+T_{1,3}+T_{1,4}.
\end{align*}

Before proving our main results, we require the analogous collection of dyadic cubes on spaces of homogeneous type as shown by Christ \cite[Theorem 11 and Lemma 15]{Chr}. This dyadic structure was independently obtained by Sawyer and Wheeden \cite{SW}.

\begin{prop}\label{pro Christ}
Let $(X,d,\mu )$ be a space
  of homogeneous type. There exists a collection of open subsets
  $\{Q_k^j\subset X: j\in \Bbb Z, k\in I_j\}$,
  where $I_j$ is a (finite or infinite) index set
  depending on $j,$ and constants $\delta \in (0, 1)$, $a_0>0$, $\eta >0$, $C_1$ and
  $C_2>0$ such that
\begin{enumerate}
\item[(i)] $\mu \big(X \setminus \bigcup \limits_{k \in I_j} Q_k^j \big)= 0$ for each fixed $j$;
\item[(ii)] $Q_k^j\cap Q_{k^\prime }^j= \varnothing \ \hbox{\hskip .05cm if} \hskip .2cm k\not= k';$
\item[(iii)] for any given $Q_k^j$ and $Q_\ell^{j^\prime }$ with $j>j^\prime$,
either $Q_k^j \subset Q_\ell^{j'}$ or $Q_k^j\cap Q_\ell^{j'}=\varnothing ;$
\item[(iv)] for each $(j, k)$ and any $j^\prime <j$, there is a unique $\ell \in I_{j'}$
such that $ Q_k^j\subset Q_\ell^{j^\prime };$
\item[(v)] for each $ Q_k^j$, $\text{\rm diam}(Q_k^j)\leq C_1\delta ^j;$
\item[(vi)] each $Q_k^j$ contains a ball $B(y_k^j, a_0\delta ^j),$ where $y_k^j \in Q_k^j;$
\item[(vii)] $\mu\{ x \in Q_k^j: d(x,X \setminus Q_k^j) \le t \delta^j\}\le C_2 t^\eta \mu(Q_k^j) \quad
  \forall j, k, \ \forall t>0.$
\end{enumerate}
\end{prop}

Properties (i) -- (iv) of Proposition \ref{pro Christ} show that all these subsets
  have the same properties as dyadic cubes in $\Bbb R^n.$ Property (v)
  implies that all these $Q_k^j$ with the same $j$ may have different measures;
  however, (v) and (vi) show that they have almost the same measures. That is, for
  each $j\in \Bbb Z,$ and $k,\ell \in I_j,$ $\mu (Q_k^j)\approx \mu (Q_\ell^j)$.
We will call these subsets $Q_k^j,$ $j\in \Bbb Z$
  and $k\in I_j,$ the {\it dyadic cubes on the spaces of homogeneous type}. 
 In fact, we can think of $Q_k^j$ as being a dyadic cube with diameter roughly $\delta^j$
  centered at $y_k^j$. As a result, we consider $CQ_k^j$ to be the dyadic cube with the same center as $Q_k^j$ and diameter $C$diam($Q_k^j$) for some constant $C$.
 
As open subset of Euclidean space has a Whitney decomposition from a system of dyadic
cubes, Seo \cite{Seo} obtained a Whitney decomposition from a system of Christ cubes.

\begin{lem}\label{WD}
Suppose that $(X,d,\mu )$ is an $A$-uniformly perfect metric space supporting a doubling metric measure, $Y$ is a closed subset of $X$,
and $\Omega=X\backslash Y$. Then $\Omega$ has a Whitney decomposition $M_{\Omega}$ satisfying the following properties:
\begin{enumerate}
\item[(1)] $\mu \big(X \setminus \bigcup \limits_{Q\in M_{\Omega}} Q \big)= 0$.
\item[(2)] \rm{diam}$(Q)\le $ \rm{dist}$(Q,Y)\le 4C_1\delta^{-1}$  \rm{diam}$(Q)$.
\item[(3)] $Q\cap Q'=\emptyset$.
\item[(4)] For any $Q\in M_\Omega$, there exists $x\in \Omega$ such that $B(x,a_0\delta^k)\subseteq Q \subseteq B(x,C_1\delta)$ for some $k$.
\end{enumerate}
The constants $\delta, a_0$ and $C_1$ are deduced from Proposition \ref{pro Christ}.
A metric space $(X,d)$ is $A$-uniformly perfect if there exists a constant $A>0$ such that
for each $x\in X$ and $0<r<\rm{diam}\,X$ there is a point $y\in X$ which satisfies $A^{-1}r\le d(x,y)\le r$.
\end{lem}
   
\subsubsection{Term $T_{1,1}$}\label{T11}
Set
\begin{align*}
\Omega_{k,1,1}&:=\left\{ x\in\R^m\backslash K: \mathsf{P}^{*}_{\mu,1,1}(\phi)(x)  >2^k \right\}.
\end{align*}

Let $\ell$ be a large constant to determined later. 
Since $(\R^m\backslash K, \|\cdot\|)$ is $1$-uniformly perfect with respect to the Lebesgue measure, we can apply the Whitney decomposition
to $\R^m\backslash K$  to get
$$\R^m\backslash K = \bigcup_{I} I,$$
where these $I$'s are dyadic cubes from Lemma \ref{WD}. 
Then we have
\begin{align*}
\int_{\R^m\backslash K} | \mathsf{P}^{*}_{\mu,1,1}(\phi)(x)|^2 \,d\sigma(x) 
&= \sum_{k\in\mathbb{Z}}\int_{(\Omega_{k+\ell,1,1}\backslash\Omega_{k+\ell+1,1,1})\cap \R^m\backslash K} | \mathsf{P}^{*}_{\mu,1,1}(\phi)(x)|^2 \,d\sigma(x) \\
&\lesssim  \sum_{k\in\mathbb{Z}} 2^{2k}\sigma\big((\Omega_{k+\ell,1,1}\backslash\Omega_{k+\ell+1,1,1})\cap \R^m\backslash K\big)\\
&=\sum_{k\in\mathbb{Z}} 2^{2k} \sum_{I\in\mathcal{I}_{k,1,1}}\sigma(I\cap (\Omega_{k+\ell,1,1}\backslash\Omega_{k+\ell+1,1,1})),
\end{align*}
where  $\mathcal{I}_{k,1,1}$ is a Whitney decomposition of $\Omega_{k,1,1}$.
Set $F_{k,1,1}(I):= I \cap (\Omega_{k+\ell,1,1}\backslash\Omega_{k+\ell+1,1,1}), I\in \mathcal{I}_{k,1,1}$ and let $\delta\in(0,1)$, to be chosen sufficiently small momentarily later.  Then we have
\begin{align*}
&\int_{\R^m\backslash K} | \mathsf{P}^{*}_{\mu,1,1}(\phi)(x)|^2 \,d\sigma(x) \\
&=\sum_{k\in\mathbb{Z}} 2^{2k} \sum_{\substack{I\in\mathcal{I}_{k,1,1}\\ \sigma(F_{k,1,1}(I))<\delta \sigma(I) }}\sigma(F_{k,1,1}(I))
+\sum_{k\in\mathbb{Z}} 2^{2k} \sum_{\substack{I\in\mathcal{I}_{k,1,1}\\ \sigma(F_{k,1,1}(I))\geq\delta \sigma(I) }}\sigma(F_{k,1,1}(I))\\
&=:A_{1,1}+B_{1,1}.
\end{align*}
As for the term $A_{1,1}$, it is obvious that 
\begin{align*}
A_{1,1}
&\leq \delta\sum_{k\in\mathbb{Z}} 2^{2k} \sum_{\substack{I\in\mathcal{I}_{k,1,1}\sigma(F_{k,1,1}(I))<\delta \sigma(I) }}\sigma(I) \\
&\leq  \delta\sum_{k\in\mathbb{Z}} 2^{2k} \sigma(\Omega_{k,1,1}\cap \R^m\backslash K) \leq  \delta \int_{\R^m\backslash K} | \mathsf{P}^{*}_{\mu,1,1}(\phi_1)(x)|^2 \,d\sigma(x),
\end{align*}
which will be absorbed into the left-hand side provided that $\delta$ is sufficiently small.  Thus it remains to show that term $B_{1,1}$ can be dominated in terms of the testing conditions.  

To continue, we first show that the operator $\mathsf{P}^{*}_{\mu,1,1}$ satisfies the following maximum principle.
\begin{lem}\label{lem P*}
There exists a positive constant $C_0$ such that
\begin{align}
\mathsf{P}^{*}_{\mu,1,1}(\phi\cdot 1_{(3\hat{I})^c} )(x) < C_02^k
\end{align}
for all $x\in I$, $I\in\mathcal{I}_{k,1,1}$ and $k\in\mathbb{Z}$.
\end{lem}
\begin{proof}
Note that $I$ are Whitney cubes, satisfying $3I \subset \Omega_k$ and 
$9C_1\delta^{-1}I\not\subset\Omega_k$. We now choose $z\in (9C_1\delta^{-1}I\cap \Omega_k^c)$. Then we obtain that
$\ell(I)<d(z,x)<5C_1\delta^{-1}\ell(I)$. Recall that $\mathsf{P}^{*}_{\mu,1,1}$ is the operator associated to the integral kernel
$ \mathsf{P}_{t,1,1}(x,y)$.
Then it is clear that  for $z\in 9C_1\delta^{-1}I\cap \Omega_{k,1,1}^c$ and for every $y$ with $(y,t)\not\in 3\hat{I}$,
there holds
\begin{align}\label{claim P1}
\mathsf{P}_{t,1,1}(x,y)\leq C_0\mathsf{P}_{t,1,1}(z,y).
\end{align}
In fact, if $y\notin 3I$, we have $d(x,y)>\ell(I)$ and hence
\begin{align*}
t+d(z,y)&\le t+d(z,x)+d(x,y) \le t+5C_1\delta^{-1}\ell(I)+d(x,y)\\
            &\le t+6C_1\delta^{-1}d(x,y);
\end{align*}
if $y\in 3I$ then $(y,t)\not\in 3\hat{I}$ gives that $t>3\ell(I)$. Hence
\begin{align*}
t+d(z,y)&\le t+d(z,x)+d(x,y) \le t+5C_1\delta^{-1}\ell(I)+d(x,y)
            \le 6C_1\delta^{-1}t+d(x,y).
\end{align*}
Therefore,
$$\mathsf{P}_{t,1,1}(x,y)=\frac{1}{t^m}\Big(\frac{t}{t+d(y,x)}\Big)^{m+ 1}\le C_0
\frac{1}{t^m}\Big(\frac{t}{t+d(y,z)}\Big)^{m+ 1}=C_0\mathsf{P}_{t,1,1}(z,y).$$
Now we multiply by $\phi(y,t) 1_{(3\hat{I})^c}$ and then integrate with respect
to $d\mu(y,t)$. As a consequence, we have
\begin{align*}
\mathsf{P}^{*}_{\mu,1,1}(\phi\cdot 1_{(3\hat{I})^c} )(x)\leq C_0 \mathsf{P}^{*}_{\mu,1,1}(\phi\cdot 1_{(3\hat{I})^c} )(z) \leq C_02^k,
\end{align*}
completing the proof of Lemma \ref{lem P*}.
\end{proof}

Now for $I\in\mathcal{I}_{k,1,1}$ with $\sigma(F_{k,1,1}(I))\geq\delta \sigma(I) $ and for each $x\in F_{k,1,1}(I)$, it follows from
the above lemma that
\begin{align*}
\mathsf{P}^{*}_{\mu,1,1}(\phi \cdot 1_{3\hat{I}} )(x) &=
\mathsf{P}^{*}_{\mu,1,1}(\phi)(x) -  \mathsf{P}^{*}_{\mu,1,1}(\phi\cdot 1_{(3\hat{I})^c} )(x) 
\geq 2^{k+\ell} - C_02^k,
\end{align*}
where $C_0$ is the constant from Lemma \ref{lem P*}.
By choosing $\ell$ such that $2^\ell>C_0+1$, we obtain that 
\begin{align*}
\mathsf{P}^{*}_{\mu,1,1}(\phi\cdot 1_{3\hat{I}} )(x) \geq 2^k. 
\end{align*}
Hence,
\begin{align*}
2^k &\leq \frac{1}{\sigma(F_{k,1,1}(I))} \int _{F_{k,1,1}(I)} \mathsf{P}^{*}_{\mu,1,1}(\phi_1\cdot 1_{3\hat{I}} )(x)\,d\sigma(x) \\
&=\frac{1}{\sigma(F_{k,1,1}(I))}  \int_{3\hat{I}} \mathsf{P}_{\sigma,1,1}( 1_{F_{k,1,1}(I)} )(x,t) \phi(x,t) \,d\mu(x,t)\\
&=\frac{1}{\sigma(F_{k,1,1}(I))}  \int_{3\hat{I}\backslash \hat{\Omega}_{k+\ell+1,1,1} } \mathsf{P}_{\sigma,1,1}( 1_{F_{k,1,1}(I)} )(x,t) \phi(x,t) \,d\mu(x,t)\\
&\quad+\frac{1}{\sigma(F_{k,1,1}(I))}  \int_{3\hat{I} \cap  \hat{\Omega}_{k+\ell+1,1,1}} \mathsf{P}_{\sigma,1,1}( 1_{F_{k,1,1}(I)} )(x,t) \phi(x,t) \,d\mu(x,t)\\
&=: B_{1,1,1}(k,I)+B_{1,1,2}(k,I),
\end{align*}
where $\mathsf{P}_{\sigma,1,1}$ is the dual operator of $\mathsf{P}^{*}_{\mu,1,1}$.  And so we now have
\begin{align*}
B_{1,1}
&\leq 2\sum_{k\in\mathbb{Z}}  \sum_{\substack{I\in\mathcal{I}_{k,1,1}\\ \sigma(F_{k,1,1}(I))\geq\delta \sigma(I) }}
B_{1,1,1}(k,I)^2\sigma(F_{k,1,1}(I)) \\
&\qquad +2\sum_{k\in\mathbb{Z}}  \sum_{\substack{I\in\mathcal{I}_{k,1,1}\\ \sigma(F_{k,1,1}(I))\geq\delta \sigma(I) }}
B_{1,1,2}(k,I)^2\sigma(F_{k,1,1}(I))\\
&=:B_{1,1,1}+B_{1,1,2}.
\end{align*}
We seek to prove that:
\begin{equation*}
B_{1,1,1}+B_{1,1,2}\lesssim \left(\mathcal{F}^2+\mathcal{B}^2\right)\left\Vert \phi\right\Vert_{L^2(M_{+};\mu)}^2.
\end{equation*}
And this will be accomplished by showing:
\begin{eqnarray}
B_{1,1,1} & \lesssim & \delta^{-2}\mathcal{F}^2\left\Vert \phi\right\Vert_{L^2(M_{+};\mu)}^2; \label{e:B1estimate}\\
B_{1,1,2} & \lesssim & \delta^{-2}\left(\mathcal{F}^2+\mathcal{B}^2\right)\left\Vert \phi\right\Vert_{L^2(M_{+};\mu)}^2. \label{e:B2estimate}
\end{eqnarray}
Recall that $\delta$ is some fixed small number and so \eqref{e:B1estimate} and \eqref{e:B2estimate} imply the desired result.

We now consider the term $B_{1,1,1}$. For $B_{1,1,1}(k,I)$, by noting that $\sigma(I)\geq \sigma(F_{k,1,1}(I))\geq\delta \sigma(I) $ and that the operator $\mathsf{P}_{\sigma,1,1}$ is a positive operator and for a positive function $f$,
 we have $ \mathsf{P}_{\sigma,1,1}(f)(x,t)\ls \mathsf{P}_{\sigma}(f)(x,t)$.
Hence,
 \begin{align*}
&B_{1,1,1}(k,I)\\
&\leq \delta^{-1}\frac{1}{\sigma(I)}  \int_{3\hat{I}\backslash \hat{\Omega}_{k+\ell+1,1,1} } \mathsf{P}_{\sigma,1,1}( 1_{I} )(x,t) \phi(x,t) \,d\mu(x,t)\\
&\ls  \delta^{-1}\frac{1}{\sigma(I)} \left(\int_{3\hat{I}\backslash \hat{\Omega}_{k+\ell+1,1,1} } |\mathsf{P}_{\sigma}( 1_{I} )(x,t)|^2  \,d\mu(x,t)\right)^{\frac{1}{2}}\left(\int_{3\hat{I}\backslash \hat{\Omega}_{k+\ell+1,1,1} } |\phi(x,t)|^2  \,d\mu(x,t)\right)^{\frac{1}{2}}  \\
&\leq \delta^{-1}\mathcal{F} \frac{1}{\sigma(I)^{\frac{1}{2}}} \left(\int_{3\hat{I}\backslash \hat{\Omega}_{k+\ell+1,1,1} } |\phi(x,t)|^2  \,d\mu(x,t)\right)^{\frac{1}{2}},
\end{align*}
where the last inequality follows from the forward testing condition for $\mathsf{P}_{t}$. Hence,
\begin{align*}
B_{1,1,1}&\leq 2\delta^{-2}\mathcal{F}^2\sum_{k\in\mathbb{Z}}  \sum_{\substack{I\in\mathcal{I}_{k,1,1}\\ \sigma(F_{k,1,1}(I))\geq\delta \sigma(I) }}
 \frac{1}{\sigma(I)} \int_{3\hat{I}\backslash \hat{\Omega}_{k+\ell+1,1,1} } |\phi(x,t)|^2  \,d\mu(x,t)\sigma(F_{k,1,1}(I))\\
&\leq 2\delta^{-2}\mathcal{F}^2
  \int_{M_{+}} |\phi(x,t)|^2 \sum_{k\in\mathbb{Z}}  \sum_{\substack{I\in\mathcal{I}_{k,1,1}\\ \sigma(F_{k,1,1}(I))\geq\delta \sigma(I) }} 
  1_{3\hat{I}\backslash \hat{\Omega}_{k+\ell+1,1,1}}(x,t)  \,d\mu(x,t)\\
&\lesssim \delta^{-2}\mathcal{F}^2
  \int_{M_{+}} |\phi(x,t)|^2   \,d\mu(x,t),
\end{align*}
where the last inequality follows from the fact that
$$ \left\|\sum_{k\in\mathbb{Z}}  \sum_{\substack{I\in\mathcal{I}_{k,1,1}\\ \sigma(F_{k,1,1}(I))\geq\delta \sigma(I) }} 
  1_{3\hat{I}\backslash \hat{\Omega}_{k+\ell+1,1,1}}(x,t)\right\|_{\infty} \lesssim 1, $$
which is a consequence of the bounded overlaps of the Whitney cubes.  Thus, we have that  $B_{1,1,1}\lesssim \delta^{-2}\mathcal{F}^2 \left\Vert \phi\right\Vert_{L^2(M_{+}^2;\mu)}^2$ proving \eqref{e:B1estimate}.

We now estimate the term $B_{1,1,2}$, which is bounded by
\begin{align}\label{B2 e1}
 \sum_{k\in\mathbb{Z}}  \sum_{\substack{I\in\mathcal{I}_{k,1,1}\\ \sigma(F_{k,1,1}(I))\geq\delta \sigma(I) }}
\frac{2\delta^{-1}}{\sigma(I)}  \left(\int_{3\hat{I} \cap  \hat{\Omega}_{k+\ell+1,1,1}} \mathsf{P}_{\sigma,1,1}( 1_{F_{k,1,1}(I)} )(x,t) \phi(x,t) \,d\mu(x,t) \right)^2. 
\end{align}
To continue, we decompose 
\begin{align}\label{B2 decom}
 3\hat{I}\cap  \hat{\Omega}_{k+\ell+1,1,1}=\bigcup_{J} \{ \hat{J}: J\subset 3I, J\in \mathcal{I}_{k+\ell+1,1,1}\}.
\end{align}
Note that for such $J$, $3J\cap F_{k,1,1}(I)=\emptyset$. 
Moreover, we have that for $(x,t)\in\hat{J}$,
\begin{align}\label{claim P}
\mathsf{P}_{\sigma,1,1}(1_{F_{k,1,1}(I)})(x,t) \approx \frac{t}{\ell(J)}\mathsf{P}_{\sigma,1,1}(1_{F_{k,1,1}(I)})(x_J,\ell(J)), 
\end{align}
where the implicit constants are independent of $x$, $t$ and $I$.  From \eqref{claim P} we obtain that
\begin{align}
&\int_{\hat{J}}\mathsf{P}_{\sigma,1,1}(1_{F_{k,1,1}(I)})(x,t)\phi(x,t)\,d\mu(x,t)\nonumber \\
&\approx  
\mathsf{P}_{\sigma,1,1}(1_{F_{k,1,1}(I)})(x_J,\ell(J)) \int_{\hat{J}} \frac{t}{\ell(J)} \phi(x,t)\,d\mu(x,t)\nonumber \\
&\approx  
\mathsf{P}_{\sigma,1,1}(1_{F_{k,1,1}(I)})(x_J,\ell(J)) \int_{\hat{J}} \frac{1}{t\ell(J)} \phi(x,t)\,d\tilde{\mu}(x,t) \nonumber\\
&\approx  
\int_{\hat{J}}  \mathsf{P}_{\sigma,1,1}(1_{F_k(I)})(x,t)  \,d\tilde{\mu}(x,t) \cdot \frac{1}{\tilde{\mu}(\hat{J})} \cdot \frac{1}{\ell(J)}\int_{\hat{J}} \frac{1}{t} \phi(x,t)\,d\tilde{\mu}(x,t) \nonumber\\
&\lesssim \int_{\hat{J}}  \mathsf{P}_{\sigma,1,1}(1_{I})(x,t) \frac{1}{t} d\tilde{\mu}(x,t) \cdot \frac{1}{\tilde{\mu}(\hat{J})} \cdot\int_{\hat{J}} \frac{1}{t} \phi(x,t)\,d\tilde{\mu}(x,t), \label{B2 e2}
\end{align}
where the last inequality follows from the fact that $\mathsf{P}_{\sigma,1,1}$ is a positive operator, and $d\tilde{\mu}(x,t) =
t^2\,d\mu(x,t)$.

From \eqref{B2 e1}, the decomposition \eqref{B2 decom} and the inequality \eqref{B2 e2}, we get that
\begin{align*}
B_{1,1,2}&\leq C\delta^{-1}\sum_{k\in\mathbb{Z}}  \sum_{\substack{I\in\mathcal{I}_{k,1,1}\\ \sigma(F_{k,1,1}(I))\geq\delta \sigma(I) }}\frac{1}{\sigma(I)}\\
&\qquad  \left(\sum_{\substack{ J\in \mathcal{I}_{k+\ell+1,1,1}\\ J\subset 3I }} \int_{\hat{J}}  \mathsf{P}_{\sigma,1,1}(1_{I})(x,t) \,\frac{d\tilde{\mu}(x,t)}{t}  \cdot \frac{\int_{\hat{J}} \frac{1}{t} \phi(x,t)d\tilde{\mu}(x,t)}{\tilde{\mu}(\hat{J})}  \right)^2. 
\end{align*}
We now define $$\alpha(J)= \frac{1}{\tilde{\mu}(\hat{J})} \int_{\hat{J}} \frac{1}{t} \phi(x,t)d\tilde{\mu}(x,t)$$ for every cube $J \subset \R^m\backslash K$. Since $\phi_1\in L^2(M_{+};\mu)$, we have
\begin{align*}
\alpha(J)&\leq \left( \frac{1}{\tilde{\mu}(\hat{J})} \int_{\hat{J}} \Big|\frac{1}{t} \phi(x,t)\Big|^2d\tilde{\mu}(x,t)\right)^{\frac{1}{2}}
\leq \left( \frac{1}{\tilde{\mu}(\hat{J})} \int_{M_{+}} \Big| \phi(x,t)\Big|^2\,d\mu(x,t)\right)^{\frac{1}{2}}\\
&=\frac{1}{\tilde{\mu}(\hat{J})^{\frac{1}{2}}} \|\phi\|_{L^2(M_{+};\mu)}. 
\end{align*}
Hence $\alpha(J)$ is well-defined for each $J$.

We now define the set $\mathcal{G}$ of principal cubes as follows. Initialize $\mathcal{G}$
to be $I_0$, which is the large dyadic cube that $\sigma$ is supported on.  Next, consider the children $J$ of 
$I_0$.  If $\alpha(J) \geq 10 \alpha(I_0)$, then add $J$ to $\mathcal{G}$. If $\alpha(J) < 10 \alpha(I_0)$, then we continue to look at the children of this $J$. Then the set $\mathcal{G}$ is defined via induction.

Next we consider the maximal function 
$$M_{\tilde{\mu}}\psi(x,t) = \sup_{J\in \mathcal{D}, (x,t)\in \hat{J}} \frac{1}{\tilde{\mu}(\hat{J})} \int_{\hat{J}}  |\psi(y,s)|d\tilde{\mu}(y,s) $$
and prove that $M_{\tilde{\mu}}$ is bounded on $L^2((\R^m\backslash K)\times(0,\infty);\tilde{\mu})$.

\begin{prop} \label{max m}
${M}_{\tilde{\mu}}:\ L^2((\R^m\backslash K)\times(0,\infty)) \to L^2((\R^m\backslash K)\times(0,\infty);\tilde{\mu}).$
\end{prop}
\begin{proof}

It is easy to see that the maximal function $M_{\tilde{\mu}}\psi(x,t)$ is bounded on 
$L^\infty((\R^m\backslash K)\times(0,\infty))$. Thus, it suffices to show that it is also weak type (1,1).

To see this, let $0\le \psi\in L^1((\R^m\backslash K)\times(0,\infty))$ and $\lambda>0$.
Consider the level set 
$$S_\lambda:=\{(x,t)\in (\R^m\backslash K)\times(0,\infty) : M_{\tilde{\mu}}\psi(x,t)>\lambda\},$$
which is the union of the maximal dyadic cubes $\hat J=J\times [0, \ell(J)]$ in $(\R^m\backslash K)\times(0,\infty)$ with some $J\in \mathcal{D}$ such that
$$\int_{\hat J} |\psi(y,s)|d\tilde{\mu}(y,s)>\lambda \tilde{\mu}(\hat J)>0.$$
Here, the argument $\hat J$ is maximal means that if there is a $J_1\in \mathcal{D}$
with $J\subsetneq J_1$, then 
$$\int_{\hat {J_1}} |\psi(y,s)|d\tilde{\mu}(y,s)\le \lambda \tilde{\mu}(\hat {J_1}).$$
We point out that such maximal dyadic cubes always exist.
In fact, suppose there is $(x,t)\in S_\lambda$ such that there is no maximal dyadic cubes in those dyadic cubes that contain $(x,t)$. There we have a sequence of increasing nested dyadic cubes $\hat J_k$ containing $(x,t)$ such that $\tilde{\mu}(\hat J_k)\to \infty$ as $k\to \infty$ with 
$$\int_{\hat J_k} |\psi(y,s)|d\tilde{\mu}(y,s)>\lambda \tilde{\mu}(\hat J_k).$$
However, this leads to contradiction since
$$\int_{\hat J_k} |\psi(y,s)|d\tilde{\mu}(y,s)\le \|\psi\|_{L^1((\R^m\backslash K)\times(0,\infty))}, \qquad \mbox{for all } k\in \mathbb Z.$$
  Thus, we have a sequence of disjoint dyadic maximal cubes $\{\hat J_k\}_{k\in \mathbb Z}$
 such that $$S_\lambda\subset \bigcup_k \hat J_k.$$
 We then have
 $$\sum_{k\in \mathbb Z} \tilde{\mu}(\hat J_k) \le \frac1\lambda \sum_{k\in \mathbb Z}
  \int_{\hat J_k} |\psi(y,s)|d\tilde{\mu}(y,s) \le \frac1\lambda 
  \int_{(\R^m\backslash K)\times(0,\infty)} |\psi(y,s)|d\tilde{\mu}(y,s)<\infty.$$
  As a consequence, we obtain that
  $$\tilde{\mu}(S_\lambda)\le \sum_{k\in \mathbb Z} \tilde{\mu}(\hat J_k)\le \frac1\lambda
  \|\psi\|_{L^1((\R^m\backslash K)\times(0,\infty))},$$
 which implies that ${M}_{\tilde{\mu}}$ is weak type (1,1), and hence the proof is complete.
\end{proof}

From the  $L^2((\R^m\backslash K)\times(0,\infty);\tilde{\mu})$-boundedness of $M_{\tilde{\mu}}$, we have
\begin{align}\label{maximal function}
\sum_{I\in\mathcal{G}} \alpha(I)^2 \tilde{\mu}(\hat{I}) &\leq  \sum_{I\in\mathcal{G}} \left( \inf_{(x,t)\in \hat{I}} M_{\tilde{\mu}}(\tilde{\phi})(x,t)\right)^2 \tilde{\mu}(\hat{I}) \\
&\leq \int_{3\hat{I}_0} M_{\tilde{\mu}}(\tilde{\phi})(x,t)^2 \,d\tilde{\mu}(x,t)\nonumber\\
&\lesssim \int_{3\hat{I}_0} \tilde{\phi}(x,t)^2 \,d\tilde{\mu}(x,t)\nonumber\\
&\leq  \|\phi\|_{L^2(M_{+},\,d\mu)}^2, \nonumber
\end{align}
where $\tilde{\phi}(x,t) = t^{-1} \phi(x,t)$.
\bigskip

Next, in the sum over $\mathcal{I}_{k+\ell+1,1,1}$, we denote $I_1=I$ and $I_i, 2\le i\le m^3$ with $I_i\cap I\not= \emptyset$ and $\overset \circ {I_i}\cap \overset \circ I=\emptyset$. The union of these intervals is $3I$. This notation, together with the definition of $\mathcal{G}$,  gives
\begin{align*}
B_{1,1,2}&\lesssim \delta^{-1}\sum_{k\in\mathbb{Z}}  \sum_{\substack{I\in\mathcal{I}_{k,1,1}\\ \sigma(F_k(I))\geq\delta \sigma(I) }}
\frac{1}{\sigma(I)} \\
&\quad\quad\times \left(\sum_{\theta=1}^{m^3} \sum_{\substack{ J\in \mathcal{I}_{k+\ell+1,1,1}\\ J\subset I_\theta, \pi_\mathcal{G}J=\pi_\mathcal{G} I_\theta}} \int_{\hat{J}}  \mathsf{P}_{\sigma,1,1}(1_{I})(x,t) \,\frac{d\tilde{\mu}(x,t)}{t}  \cdot \frac{\int_{\hat{J}} \frac{1}{t} \phi(x,t)d\tilde{\mu}(x,t)}{\tilde{\mu}(\hat{J})}  \right)^2\\
&\quad\quad \quad+\delta^{-1}\sum_{k\in\mathbb{Z}}  \sum_{\substack{I\in\mathcal{I}_{k,1,1}\\ \sigma(F_k(I))\geq\delta \sigma(I) }}
\frac{1}{\sigma(I)} \\
&\quad\quad\quad\quad\times \left(\sum_{\theta=1}^{m^3} \sum_{\substack{ J\in \mathcal{I}_{k+\ell+1,1,1}\\ J\subset I_\theta, \pi_\mathcal{G}J\subsetneq\pi_\mathcal{G} I_\theta}} \int_{\hat{J}}  \mathsf{P}_{\sigma,1,1}(1_{I})(x,t) \frac{1}{t} \,d\tilde{\mu}(x,t) \cdot \frac{\int_{\hat{J}} \frac{1}{t} \phi(x,t)\,d\tilde{\mu}(x,t) }{\tilde{\mu}(\hat{J})} \right)^2 \\
&=: B_{1,1,21}+B_{1,1,22}.
\end{align*}
\color{black}

Thus, to prove \eqref{e:B2estimate} it will suffice to provide an estimate of the right form on each of $B_{1,1,21}$ and $B_{1,1,22}$.  We will show that:
\begin{eqnarray}
B_{1,1,21} & \lesssim & \delta^{-2} \mathcal{B}^2  \|\phi\|_{L^2(M_{+};\mu)}^2;\label{e:B21estimate}\\
B_{1,1,22} & \lesssim & \delta^{-1}\mathcal{F}^2\|\phi\|_{L^2(M_{+};\mu)}^2.\label{e:B22estimate}
\end{eqnarray}

For $B_{1,1,21}$, using the definition of $\alpha(J)$, we have
\begin{align*}
&B_{1,1,21}\\
&\lesssim \delta^{-1}\sum_{\theta=1}^{m^3}\sum_{k\in\mathbb{Z}}  \sum_{\substack{I\in\mathcal{I}_{k,1,1}\\ \sigma(F_{k,1,1}(I))\geq\delta \sigma(I) }}
\frac{1}{\sigma(I)}  \alpha(J)^2\left( \sum_{\substack{ J\in \mathcal{I}_{k+\ell+1,1,1}\\ J\subset I_\theta, \pi_\mathcal{G}J=\pi_\mathcal{G} I_\theta}} \int_{\hat{J}}  \mathsf{P}_{\sigma,1,1}(1_{I})(x,t) \,\frac{d\tilde{\mu}(x,t)}{t}   \right)^2\\
&\lesssim \delta^{-1}\sum_{\theta=1}^{m^3}\sum_{k\in\mathbb{Z}}  \sum_{\substack{I\in\mathcal{I}_{k,1,1}\\ \sigma(F_{k,1,1}(I))\geq\delta \sigma(I) }}
\frac{1}{\sigma(I)}  \alpha(\pi_{\mathcal{G}}I_\theta)^2\left( \sum_{\substack{ J\in \mathcal{I}_{k+\ell+1,1,1}\\ J\subset I_\theta, \pi_\mathcal{G}J=\pi_\mathcal{G} I_\theta}} \int_{\hat{J}}  \mathsf{P}_{\sigma,1,1}(1_{I})(x,t) t \,d\mu(x,t)  \right)^2\\
&\lesssim \delta^{-1}\sum_{\theta=1}^{m^3}\sum_{k\in\mathbb{Z}}  \sum_{\substack{I\in\mathcal{I}_{k,1,1}\\ \sigma(F_{k,1,1}(I))\geq\delta \sigma(I) }}
\frac{1}{\sigma(I)}  \alpha(\pi_{\mathcal{G}}I_\theta)^2\left( \int_{I}  \mathsf{P}^{*}_{\mu,1,1}(t 1_{\hat{I_\theta}})(y)  \,d\sigma(y)  \right)^2\\
&\lesssim \delta^{-1}\sum_{\theta=1}^{m^3}\sum_{k\in\mathbb{Z}}  \sum_{\substack{I\in\mathcal{I}_{k,1,1}\\ \sigma(F_{k,1,1}(I))\geq\delta \sigma(I) }}
\frac{1}{\sigma(I)}  \alpha(\pi_{\mathcal{G}}I_\theta)^2\, \sigma(I)   \int_{3I}  \mathsf{P}^{*}_{\mu}(t 1_{\hat{I_\theta}})(y)^2  \,d\sigma(y)    \\
&\lesssim\delta^{-1}\mathcal{B}^2\sum_{\theta=1}^{m^3}\sum_{k\in\mathbb{Z}}  \sum_{\substack{I\in\mathcal{I}_{k,1,1}\\ \sigma(F_{k,1,1}(I))\geq\delta \sigma(I) }}
  \alpha(\pi_{\mathcal{G}}I_\theta)^2\,    \tilde{\mu}(\hat{I}_\theta)   \\
&= \delta^{-1}\mathcal{B}^2\sum_{\theta=1}^{m^3}  \sum_{G\in\mathcal{G}}  \alpha(G)^2  \sum_{k\in\mathbb{Z}}  \sum_{\substack{I\in\mathcal{I}_{k,1,1}\\ \sigma(F_{k,1,1}(I))\geq\delta \sigma(I)\\  \pi_{\mathcal{G}} I_\theta =G}}
  \tilde{\mu}(\hat{I}_\theta),
\end{align*}
where the last inequality follows from the testing condition for $\mathsf{P}^{*}$.

We point out that for each dyadic cube $I$, the set
$$ \left\{k\in\mathbb{Z}:\ I\in\mathcal{I}_{k,1,1}, \sigma(F_{k,1,1}(I)) \geq \delta \sigma(I)\right\} $$
consists of at most $\delta^{-1}$ consecutive integers. Actually, that the integers
in this set are consecutive follows from the nested property of the collections $\mathcal{I}_{k,1,1}$.
Moreover, note that for each fixed $I$, the sets $F_{k,1,1}(I)\subset I$ are pairwise disjoint (with respect to $k$), and 
for each $k$,  $\sigma(F_{k,1,1}(I)) \geq \delta \sigma(I)$. Hence, there are at most $\delta^{-1}$ such integers $k$.

As a consequence, we obtain that
\begin{align*}
B_{1,1,21}&\leq C\delta^{-2} \mathcal{B}^2\sum_{G\in\mathcal{G}}  \alpha(G)^2  
  \tilde{\mu}(\hat{G})\leq C\delta^{-2} \mathcal{B}^2  \|\phi\|_{L^2(M_{+};\mu)}^2,
\end{align*}
where the last inequality follows from the maximal inequality \eqref{maximal function}.  This gives \eqref{e:B21estimate}.

We now turn to the estimate $B_{1,1,22}$. Using the definition of $\alpha(J)$, we have
\begin{align*}
&B_{1,1,22}\\
&\lesssim\delta^{-1}\sum_{\theta=1}^{m^3} \sum_{k\in\mathbb{Z}}  \sum_{\substack{I\in\mathcal{I}_{k,1,1}\\ \sigma(F_{k,1,1}(I))\geq\delta \sigma(I) }}
\frac{1}{\sigma(I)}  \left(\sum_{\substack{ J\in \mathcal{I}_{k+\ell+1,1,1}\\ J\subset I_\theta\\ \pi_\mathcal{G}J\subsetneq\pi_\mathcal{G} I_\theta}} \int_{\hat{J}}  \mathsf{P}_{\sigma,1,1}(1_{I})(x,t) \,\frac{d\tilde{\mu}(x,t)}{t}  \cdot \alpha(J) \right)^2 \\
&\lesssim\delta^{-1}\sum_{\theta=1}^{m^3} \sum_{k\in\mathbb{Z}}  \sum_{\substack{I\in\mathcal{I}_{k,1,1}\\ \sigma(F_{k,1,1}(I))\geq\delta \sigma(I) }}
\frac{1}{\sigma(I)}  \sum_{\substack{ J\in \mathcal{I}_{k+\ell+1,1,1}\\ J\subset I_\theta\\ \pi_\mathcal{G}J\subsetneq\pi_\mathcal{G} I_\theta}} \left[\int_{\hat{J}}  \mathsf{P}_{\sigma,1,1}(1_{I})(x,t)\, \frac{d\tilde{\mu}(x,t)}{t} \right]^{2}\tilde{\mu}(\hat{J})^{-1}  \\
&\quad\quad\times  \sum_{\substack{ J\in \mathcal{I}_{k+\ell+1,1,1}\\ J\subset I_\theta\\ \pi_\mathcal{G}J\subsetneq\pi_\mathcal{G} I_\theta}} \tilde{\mu}(\hat{J}) \alpha(J)^2,
\end{align*}
where the last inequality follows from Cauchy--Schwarz inequality. Next, from the Cauchy--Schwarz inequality, the definition of $\tilde{\mu}$ and the testing condition, we have
\begin{align*}
  \sum_{\substack{ J\in \mathcal{I}_{k+\ell+1,1,1}\\ J\subset I_\theta\\ \pi_\mathcal{G}J\subsetneq\pi_\mathcal{G} I_\theta}} \left[\int_{\hat{J}}  \mathsf{P}_{\sigma,1,1}(1_{I})(x,t) \,\frac{d\tilde{\mu}(x,t)}{t} \right]^{2}\tilde{\mu}(\hat{J})^{-1}   & \leq \sum_{\substack{ J\in \mathcal{I}_{k+\ell+1,1,1}\\ J\subset I_\theta\\ \pi_\mathcal{G}J\subsetneq\pi_\mathcal{G} I_\theta}} \int_{\hat{J}}  \mathsf{P}_\sigma(1_{I})(x,t)^2  \,d\mu(x,t)\\
& \leq \mathcal{F}^2 \sigma(I),
\end{align*}
which implies that
\begin{align}
B_{1,1,22}
&\lesssim\mathcal{F}^2\delta^{-1}\sum_{\theta=1}^{m^3} \sum_{k\in\mathbb{Z}}  \sum_{\substack{I\in\mathcal{I}_{k,1,1}\\ \sigma(F_{k,1,1}(I))\geq\delta \sigma(I) }}
\frac{1}{\sigma(I)}   \sigma(I)  \sum_{\substack{ J\in \mathcal{I}_{k+\ell+1,1,1}\\ J\subset I_\theta\\ \pi_\mathcal{G}J\subsetneq\pi_\mathcal{G} I_\theta}} \tilde{\mu}(\hat{J}) \alpha(J)^2\nonumber\\
&\lesssim\mathcal{F}^2\delta^{-1} \sum_{k\in\mathbb{Z}}  \sum_{\substack{I\in\mathcal{I}_{k,1,1}\\ \sigma(F_{k,1,1}(I))\geq\delta \sigma(I) }}
\sum_{\substack{ J\in \mathcal{I}_{k+\ell+1,1,1}\\ J\subset I_\theta, \pi_\mathcal{G}J\subsetneq\pi_\mathcal{G} I_\theta}} \tilde{\mu}(\hat{J}) \alpha(\pi_\mathcal{G}J)^2.\label{B22}
\end{align}

We now recall a technical result from Lacey \cite{L}*{Lemma 8.15}.
\begin{lem}\label{L}
There is an absolute constant $C$ such that for any $G\in\mathcal{G}$, the cardinality of the set
$$ \left\{k:\ \pi_{\mathcal{G}}J = G,\ J\in\mathcal{I}_{k+\ell+1,1,1}\ {\rm contributes\ to\ the\ } k{\rm th\ sum\ in\ } \eqref{B22} \right\} $$
is at most $C$.
\end{lem}
As a consequence of Lemma \ref{L}, we get that 
\begin{align*}
B_{1,1,22}
&\lesssim \mathcal{F}^2\delta^{-1} \sum_{I\in \mathcal{G}} \tilde{\mu}(\hat{I}) \alpha(I)^2 \lesssim\mathcal{F}^2\delta^{-1}  \|\phi\|_{L^2(M_{+};\mu)}^2
\end{align*}
which is \eqref{e:B22estimate}.

\subsubsection{Term $T_{1,2}$}
Set
\begin{align*}
\Omega_{k,1,2}&:=\left\{ x\in \R^m\backslash K: \mathsf{P}^{*}_{\mu,1,2}(\phi_1)(x)  >2^k \right\}.
\end{align*}

Let $\ell_1$ be a large constant to determined later. Then we apply the Whitney decomposition
to $\R^m\backslash K$  to get
$$\R^m\backslash K = \bigcup_{I} I ,$$
where these $I$'s are dyadic cubes from Lemma \ref{WD}. 
Then we have
\begin{align*}
\int_{\R^m\backslash K} | \mathsf{P}^{*}_{\mu,1,2}(\phi_1)(x)|^2 \,d\sigma(x) 
&= \sum_{k\in\mathbb{Z}}\int_{(\Omega_{k+\ell_1,1,2}\backslash\Omega_{k+\ell_1+1,1,2})\cap \R^m\backslash K} | \mathsf{P}^{*}_{\mu,1,2}(\phi_1)(x)|^2 \,d\sigma(x) \\
&\lesssim  \sum_{k\in\mathbb{Z}} 2^{2k}\sigma\big((\Omega_{k+\ell_1,1,2}\backslash\Omega_{k+\ell_1+1,1,2})\cap \R^m\backslash K\big)\\
&=\sum_{k\in\mathbb{Z}} 2^{2k} \sum_{I\in\mathcal{I}_{k,1,2
}}\sigma(I\cap (\Omega_{k+\ell_1,1,2}\backslash\Omega_{k+\ell_1+1,1,2})),
\end{align*}
where $\mathcal{I}_{k,1,2}$ is a Whitney decomposition of $\Omega_{k,1,2}$.
Set $F_{k,1,2}(I):= I \cap (\Omega_{k+\ell_1,1,2}\backslash\Omega_{k+\ell_1+1,1,2}), I\in \mathcal{I}_{k,1,2}$. Now let $\delta_1\in(0,1)$, to be chosen sufficiently small.  Then we have
\begin{align*}
&\int_{\R^m\backslash K} | \mathsf{P}^{*}_{\mu,1,2}(\phi)(x)|^2 \,d\sigma(x) \\
&=\sum_{k\in\mathbb{Z}} 2^{2k} \sum_{\substack{I\in\mathcal{I}_{k,1,2}\\ \sigma(F_{k,1,2}(I))<\delta_1 \sigma(I) }}\sigma(F_{k,1,2}(I))
+\sum_{k\in\mathbb{Z}} 2^{2k} \sum_{\substack{I\in\mathcal{I}_{k,1,2}\\ \sigma(F_{k,1,2}(I))\geq\delta_1 \sigma(I) }}\sigma(F_{k,1,2}(I))\\
&=:A_{1,2}+B_{1,2}.
\end{align*}
As for the term $A_{1,2}$, it is handled in the same fashion as $A_{1,1}$. Thus it remains to show that term $B_{1,2}$ can be dominated in terms of the testing conditions.

To continue, we first show that the operator $\mathsf{P}^{*}_{\mu,1,2}$ satisfies the following maximum principle.
\begin{lem}\label{lem P*12}
There exists a positive constant $C_1$ such that
\begin{align}
\mathsf{P}^{*}_{\mu,1,2}(\phi\cdot 1_{(3\hat{I})^c} )(x) < C_12^k
\end{align}
for all $x\in I$, $I\in\mathcal{I}_k$ and $k\in\mathbb{Z}$.
\end{lem}
\begin{proof}
Note that $I$ is the Whitney cubes, satisfying $3I\subset \Omega_k$ and 
$9C_1\delta^{-1}I\not\subset\Omega_k$. We now choose $z\in (9C_1\delta^{-1}I\cap \Omega_k^c)$. Then we obtain that
$\ell(I)<d(z,x)<5C_1\delta^{-1}\ell(I)$. Recall that $\mathsf{P}^{*}_{\mu,1,2}$ is the operator associated to the integral kernel
$ \mathsf{P}_{t,1,2}(x,y)$.
Since $x\in I$ and $3I\subset \Omega_k\subset \R^m\backslash K$, we have $d(x,K)\ge \ell(I)$.
For $z\in 9C_1\delta^{-1}I\cap \Omega_k^c$, it is clear that
$$|z|\approx 1+d(z,k)\lesssim 1+d(z,x)+d(x,K)\lesssim 1+d(x,K) \approx |x|.$$
Hence for $z\in 9C_1\delta^{-1}I\cap \Omega_k^c$ and for every $y$ with $(y,t)\not\in 3\hat{I}$,
there holds
\begin{equation}\label{claim P12}
\begin{aligned}
\mathsf{P}_{t,1,2}(y,x)&=\frac{1}{t^n|x|^{m-2}}\Big(\frac{t}{t+d(x,y)}\Big)^{n+ 1}\\
&\leq C_1\frac{1}{t^n|z|^{m-2}}\Big(\frac{t}{t+d(z,y)}\Big)^{n+ 1}= C_1\mathsf{P}_{t,1,2}(y,z).
\end{aligned}
\end{equation}
Now we multiply it by $\phi(y,t) 1_{(3\hat{I})^c}$ and then integrate with respect
to $d\mu(y,t)$. As a consequence, we have
\begin{align*}
\mathsf{P}^{*}_{\mu,1,2}(\phi\cdot 1_{(3\hat{I})^c} )(x)\leq C_1 \mathsf{P}^{*}_{\mu,1,2}(\phi\cdot 1_{(3\hat{I})^c} )(z) \leq C_12^k.
\end{align*}
The proof  is complete.
\end{proof}

Now for $I\in\mathcal{I}_{k,1,2}$ with $\sigma(F_{k,1,2}(I))\geq\delta_1 \sigma(I) $ and for each $x\in F_{k,1,2}(I)$, it follows from
the above lemma that
\begin{align*}
\mathsf{P}^{*}_{\mu,1,2}(\phi\cdot 1_{3\hat{I}} )(x) &=
\mathsf{P}^{*}_{\mu,1,2}(\phi )(x) -  \mathsf{P}^{*}_{\mu,1,2}(\phi\cdot 1_{(3\hat{I})^c} )(x) 
\geq 2^{k+\ell_1} - C_12^k,
\end{align*}
where $C_1$ is the constant from Lemma \ref{lem P*12}.
By choosing $\ell_1$ such that $2^{\ell_1}>C_1+1$, we obtain that 
\begin{align*}
\mathsf{P}^{*}_{\mu,1,2}(\phi\cdot 1_{3\hat{I}} )(x) \geq 2^k. 
\end{align*}
Hence,
\begin{align*}
2^k &\leq \frac{1}{\sigma(F_{k,1,2}(I))} \int _{F_{k,1,2}(I)} \mathsf{P}^{*}_{\mu,1,2}(\phi\cdot 1_{3\hat{I}} )(x)\,d\sigma(x) \\
&=\frac{1}{\sigma(F_{k,1,2}(I))}  \int_{3\hat{I}\backslash \hat{\Omega}_{k+\ell_1+1,1,2} } \mathsf{P}_{\sigma,1,2}( 1_{F_{k,1,2}(I)} )(x,t) \phi(x,t) \,d\mu(x,t)\\
&\quad+\frac{1}{\sigma(F_{k,1,2}(I))}  \int_{3\hat{I} \cap  \hat{\Omega}_{k+\ell_1+1,1,2}} \mathsf{P}_{\sigma,1,2}( 1_{F_{k,1,2}(I)} )(x,t) \phi(x,t) \,d\mu(x,t)\\
&=: B_{1,2,1}(k,I)+B_{1,2,2}(k,I),
\end{align*}
where $\mathsf{P}_{\sigma,1,2}$ is the dual operator of $\mathsf{P}^{*}_{\mu,1,2}$.

Hence we obtain that 
\begin{align*}
B_{1,2}
&\leq 2\sum_{k\in\mathbb{Z}}  \sum_{\substack{I\in\mathcal{I}_{k,1,2}\\ \sigma(F_{k,1,2}(I))\geq\delta_1 \sigma(I) }}
B_{1,2,1}(k,I)^2\sigma(F_{k,1,2}(I)) \\
&\qquad +2\sum_{k\in\mathbb{Z}}  \sum_{\substack{I\in\mathcal{I}_{k,1,2}\\ \sigma(F_{k,1,2}(I))\geq\delta_1 \sigma(I) }}
B_{1,2,2}(k,I)^2\sigma(F_{k,1,2}(I))\\
&=:B_{1,2,1}+B_{1,2,2}.
\end{align*}
We claim that
\begin{eqnarray}
B_{1,2,1} & \lesssim & \delta_1^{-2}\mathcal{F}^2\left\Vert \phi\right\Vert_{L^2(M_{+};\mu)}^2; \\\label{e:B121estimate}
B_{1,2,2} & \lesssim & \delta_1^{-2}\left(\mathcal{F}^2+\mathcal{B}^2\right)\left\Vert \phi\right\Vert_{L^2(M_{+};\mu)}^2; \label{e:B122estimate}
\end{eqnarray}
and hence
\begin{equation*}
B_{1,2}\lesssim \left(\mathcal{F}^2+\mathcal{B}^2\right)\left\Vert \phi\right\Vert_{L^2(M_{+};\mu)}^2.
\end{equation*}

We now consider the term $B_{1,2,1}$. As for $B_{1,2,1}(k,I)$, by noting that $\sigma(I)\geq \sigma(F_{k,1,2}(I))\geq\delta_1 \sigma(I) $ and that the operator $\mathsf{P}_{\sigma,1,2}$ is a positive operator and for a positive function $f$,
 we have 
 $$\mathsf{P}_{\sigma,1,2}(f)(x,t)=\int_M \mathsf{P}_{t,1,2}(x,y)f(x,y)d\sigma(y)
 \ls \int_M \mathsf{P}_{t}(x,y)f(x,y)d\sigma (y)=\mathsf{P}_{\sigma}(f)(x,t).$$
Hence,
 \begin{align*}
&B_{1,2,1}(k,I)\\
&\leq \delta_1^{-1}\frac{1}{\sigma(I)}  \int_{3\hat{I}\backslash \hat{\Omega}_{k+\ell_1+1,1,2} } \mathsf{P}_{\sigma,1,2}( 1_{I} )(x,t) \phi_1(x,t) \,d\mu(x,t)\\
&\ls  \delta_1^{-1}\frac{1}{\sigma(I)} \left(\int_{3\hat{I}\backslash \hat{\Omega}_{k+\ell_1+1,1,2} } |\mathsf{P}_{\sigma}( 1_{I} )(x,t)|^2  \,d\mu(x,t)\right)^{\frac{1}{2}}\left(\int_{3\hat{I}\backslash \hat{\Omega}_{k+\ell_1+1,1,2} } |\phi
(x,t)|^2  \,d\mu(x,t)\right)^{\frac{1}{2}}  \\
&\leq \delta_1^{-1}\mathcal{F} \frac{1}{\sigma(I)^{\frac{1}{2}}} \left(\int_{3\hat{I}\backslash \hat{\Omega}_{k+\ell_1+1,1,2} } |\phi(x,t)|^2  \,d\mu(x,t)\right)^{\frac{1}{2}},
\end{align*}
where the last inequality follows from the forward testing condition for $\mathsf{P}_{t}$. Hence,
\begin{align*}
B_{1,2,1}&\leq 2\delta_1^{-2}\mathcal{F}^2\sum_{k\in\mathbb{Z}}  \sum_{\substack{I\in\mathcal{I}_{k,1,2}\\ \sigma(F_{k,1,2}(I))\geq\delta_1 \sigma(I) }}
 \frac{1}{\sigma(I)} \int_{3\hat{I}\backslash \hat{\Omega}_{k+\ell_1+1,1,2} } |\phi(x,t)|^2  \,d\mu(x,t)\sigma(F_{k,1,2}(I))\\
&\leq 2\delta_1^{-2}\mathcal{F}^2
  \int_{M_{+}} |\phi(x,t)|^2 \sum_{k\in\mathbb{Z}}  \sum_{\substack{I\in\mathcal{I}_{k,1,2}\\ \sigma(F_{k,1,2}(I))\geq\delta_1 \sigma(I) }} 
  1_{3\hat{I}\backslash \hat{\Omega}_{k+\ell_1+1,1,2}}(x,t)  \,d\mu(x,t)\\
&\lesssim \delta_1^{-2}\mathcal{F}^2
  \int_{M_{+}} |\phi(x,t)|^2   \,d\mu(x,t),
\end{align*}
where the last inequality follows from the fact that
$$ \left\|\sum_{k\in\mathbb{Z}}  \sum_{\substack{I\in\mathcal{I}_{k,1,2}\\ \sigma(F_{k,1,2}(I))\geq\delta_1 \sigma(I) }} 
  1_{3\hat{I}\backslash \hat{\Omega}_{k+\ell_1+1,1,2}}(x,t)\right\|_{\infty} \lesssim 1, $$
which is a consequence of the bounded overlaps of the Whitney cubes.  Thus, we have that  $B_{1,1,1}\lesssim \delta_1^{-2}\mathcal{F}^2 \left\Vert \phi\right\Vert_{L^2(M_{+}^2;\mu)}^2$ proving \eqref{e:B21estimate}.

We now estimate $B_{1,2,2}$, which is bounded by
\begin{align}\label{B122}
\sum_{k\in\mathbb{Z}}  \sum_{\substack{I\in\mathcal{I}_{k,1,2}\\ \sigma(F_{k,1,2}(I))\geq\delta_1 \sigma(I) }}
\frac{2\delta_1^{-1}}{\sigma(I)}  \left(\int_{3\hat{I} \cap  \hat{\Omega}_{k+\ell_1+1,1,2}} \mathsf{P}_{\sigma,1,2}( 1_{F_{k,1,2}(I)} )(x,t) \phi(x,t) \,d\mu(x,t) \right)^2. 
\end{align}
To continue, we decompose 
\begin{align}\label{B122 decom}
 3\hat{I}\cap  \hat{\Omega}_{k+\ell_1+1,1,2}=\bigcup_{J} \{ \hat{J}: J\subset 3I, J\in \mathcal{I}_{k+\ell_1+1,1,2}  \}.
\end{align}
Since $J\in \mathcal{I}_{k+\ell_1+1} $, we have $3J\subset {\Omega}_{k+\ell_1+1,1,2}\subset \R^m \backslash K$ and hence $d(x,K)\ge \ell(J)$ for any $x\in J$.
Thus, for $(x,t)\in\hat{J}$, it follows $|x|\approx |x_J|$ and $t+d(y,x)|\approx \ell(J)+d(y,x_J)$.
Moreover, we have that for $(x,t)\in\hat{J}$,
\begin{equation}\label{claim P12}
\begin{aligned}
&\mathsf{P}_{\sigma,1,2}(1_{F_{k,1,2}(I)})(x,t) \\
&=\int_{\R^m \backslash K} \frac{1}{t^n|x|^{m-2}}\Big(\frac{t}{t+d(x,y)}\Big)^{n+ 1} 1_{F_{k,1,2}(I)}(y)d\sigma y\\
&\approx  \frac{t}{\ell(J)}\int_{\R^m \backslash K} \frac{1}{\ell(J)^n|x_J|^{m-2}}\Big(\frac{\ell(J)}{\ell(J)+d(x_J,y)}\Big)^{n+ 1}  1_{F_{k,1,2}(I)}(y)d\sigma y\\
&=\frac{t}{\ell(J)}\mathsf{P}_{\sigma,1,2}(1_{F_{k,1,2}(I)})(x_J,\ell(J)), 
\end{aligned}
\end{equation}
where the implicit constants are independent of $x$, $t$ and $I$.

From \eqref{claim P12} we obtain that
\begin{align}
&\int_{\hat{J}}\mathsf{P}_{\sigma,1,2}(1_{F_{k,1,2}(I)})(x,t)\phi(x,t)\,d\mu(x,t)\nonumber \\
&\approx  
\mathsf{P}_{\sigma,1,2}(1_{F_{k,1,2}(I)})(x_J,\ell(J)) \int_{\hat{J}} \frac{t}{\ell(J)} \phi(x,t)\,d\mu(x,t)\nonumber \\
&\approx  
\mathsf{P}_{\sigma,1,2}(1_{F_{k,1,2}(I)})(x_J,\ell(J)) \int_{\hat{J}} \frac{1}{t\ell(J)} \phi(x,t)\,d\tilde{\mu}(x,t) \nonumber\\
&\approx  
\int_{\hat{J}}  \mathsf{P}_{\sigma,1,2}(1_{F_{k,1,2}(I)})(x,t)  \,d\tilde{\mu}(x,t) \cdot \frac{1}{\tilde{\mu}(\hat{J})} \cdot \frac{1}{\ell(J)}\int_{\hat{J}} \frac{1}{t} \phi(x,t)\,d\tilde{\mu}(x,t) \nonumber\\
&\lesssim \int_{\hat{J}}  \mathsf{P}_{\sigma,1,2}(1_{I})(x,t) \frac{1}{t} d\tilde{\mu}(x,t) \cdot \frac{1}{\tilde{\mu}(\hat{J})} \cdot\int_{\hat{J}} \frac{1}{t} \phi(x,t)\,d\tilde{\mu}(x,t), \label{B12 e2}
\end{align}
where the last inequality follows since $\mathsf{P}_{\sigma,1,2}$ is a positive operator, and $d\tilde{\mu}(x,t) =t^2\,d\mu(x,t)$.

From \eqref{B122}, the decomposition \eqref{B122 decom} and the inequality \eqref{B12 e2}, we get that
\begin{align*}
B_{1,2,2}
&\lesssim \delta_1^{-1}\sum_{k\in\mathbb{Z}}  \sum_{\substack{I\in\mathcal{I}_{k,1,2}\\ \sigma(F_{k,1,2}(I))\geq\delta_1 \sigma(I) }}\frac{1}{\sigma(I)}\\
 &\quad \times\left(\sum_{\substack{ J\in \mathcal{I}_{k+\ell_1+1,1,2}\\ J\subset 3I }} \int_{\hat{J}}  \mathsf{P}_{\sigma,1,2}(1_{I})(x,t) \,\frac{d\tilde{\mu}(x,t)}{t}  \cdot \frac{\int_{\hat{J}} \frac{1}{t} \phi(x,t)d\tilde{\mu}(x,t)}{\tilde{\mu}(\hat{J})}  \right)^2. 
\end{align*}
Let $\alpha(J), \mathcal{G}$ and $M_{\tilde{\mu}}$ be defined in subsection \ref{T11} .
In the sum over $\mathcal{I}_{k+\ell_1+1,1,2}$, we denote $I_1=I$ and $I_i, 2\le i\le m^3$ with $I_i\cap I\not= \emptyset$ and $\overset \circ {I_i}\cap \overset \circ I=\emptyset$. The union of these intervals is $3I$. Therefore,
\begin{align*}
B_{1,2,2}&\lesssim \delta_1^{-1}\sum_{k\in\mathbb{Z}}  \sum_{\substack{I\in\mathcal{I}_{k,1,2}\\ \sigma(F_k(I))\geq\delta_1 \sigma(I) }}
\frac{1}{\sigma(I)}  \left(\sum_{\theta=1}^{m^3} \sum_{\substack{ J\in \mathcal{I}_{k+\ell_1+1,1,2}\\ J\subset I_\theta, \pi_\mathcal{G}J=\pi_\mathcal{G} I_\theta}} \int_{\hat{J}}  \mathsf{P}_{\sigma,1,2}(1_{I})(x,t) \,\frac{d\tilde{\mu}(x,t)}{t}  \cdot \alpha(J) \right)^2\\
&\quad+\delta_1^{-1}\sum_{k\in\mathbb{Z}}  \sum_{\substack{I\in\mathcal{I}_{k,1,2}\\ \sigma(F_k(I))\geq\delta_1 \sigma(I) }}
\frac{1}{\sigma(I)}\left(\sum_{\theta=1}^{m^3} \sum_{\substack{ J\in \mathcal{I}_{k+\ell_1+1,1,2}\\ J\subset I_\theta, \pi_\mathcal{G}J\subsetneq\pi_\mathcal{G} I_\theta}} \int_{\hat{J}}  \mathsf{P}_{\sigma,1,2}(1_{I})(x,t) \frac{ \,d\tilde{\mu}(x,t)} {t}\cdot \alpha(J) \right)^2 \\
&=: B_{1,2,21}+B_{1,2,22}.
\end{align*}

To prove \eqref{e:B122estimate}, it suffices to provide the following estimates:
\begin{eqnarray}
B_{1,2,21} & \lesssim & \delta_1^{-2} \mathcal{B}^2  \|\phi\|_{L^2(M_{+};\mu)}^2;\label{e:B1221estimate}\\
B_{1,2,22} & \lesssim & \delta_1^{-1}\mathcal{F}^2\|\phi\|_{L^2(M_{+};\mu)}^2.\label{e:B1222estimate}
\end{eqnarray}
We now estimate the term $B_{1,2,21}$,
\begin{align*}
&B_{1,2,21}\\
&\lesssim \delta^{-1}\sum_{\theta=1}^{m^3}\sum_{k\in\mathbb{Z}}  \sum_{\substack{I\in\mathcal{I}_{k,1,2}\\ \sigma(F_{k,1,2}(I))\geq\delta_1 \sigma(I) }}
\frac{1}{\sigma(I)}  \alpha(J)^2\left( \sum_{\substack{ J\in \mathcal{I}_{k+\ell_1+1,1,2}\\ J\subset I_\theta, \pi_\mathcal{G}J=\pi_\mathcal{G} I_\theta}} \int_{\hat{J}}  \mathsf{P}_{\sigma,1,2}(1_{I})(x,t) \,\frac{d\tilde{\mu}(x,t)}{t}   \right)^2\\
&\lesssim \delta^{-1}\sum_{\theta=1}^{m^3}\sum_{k\in\mathbb{Z}}  \sum_{\substack{I\in\mathcal{I}_{k,1,2}\\ \sigma(F_{k,1,2}(I))\geq\delta_1 \sigma(I) }}
\frac{1}{\sigma(I)}  \alpha(\pi_{\mathcal{G}}I_\theta)^2\left( \sum_{\substack{ J\in \mathcal{I}_{k+\ell_1+1,1,2}\\ J\subset I_\theta, \pi_\mathcal{G}J=\pi_\mathcal{G} I_\theta}} \int_{\hat{J}}  \mathsf{P}_{\sigma,1,2}(1_{I})(x,t) t \,d\mu(x,t)  \right)^2\\
&\lesssim \delta^{-1}\sum_{\theta=1}^{m^3}\sum_{k\in\mathbb{Z}}  \sum_{\substack{I\in\mathcal{I}_{k,1,2}\\ \sigma(F_{k,1,2}(I))\geq\delta_1 \sigma(I) }}
\frac{1}{\sigma(I)}  \alpha(\pi_{\mathcal{G}}I_\theta)^2\left( \int_{I}  \mathsf{P}^{*}_{\mu,1,2}(t 1_{\hat{I_\theta}})(y)  \,d\sigma(y)  \right)^2\\
&\lesssim \delta^{-1}\sum_{\theta=1}^{m^3}\sum_{k\in\mathbb{Z}}  \sum_{\substack{I\in\mathcal{I}_{k,1,2}\\ \sigma(F_{k,1,2}(I))\geq\delta_1 \sigma(I) }}
\frac{1}{\sigma(I)}  \alpha(\pi_{\mathcal{G}}I_\theta)^2\, \sigma(I)   \int_{3I}  \mathsf{P}^{*}_{\mu}(t 1_{\hat{I_\theta}})(y)^2  \,d\sigma(y)    \\
&\lesssim\delta^{-1}\mathcal{B}^2\sum_{\theta=1}^{m^3}\sum_{k\in\mathbb{Z}}  \sum_{\substack{I\in\mathcal{I}_{k,1,2}\\ \sigma(F_{k,1,2}(I))\geq\delta_1 \sigma(I) }}
  \alpha(\pi_{\mathcal{G}}I_\theta)^2\,    \tilde{\mu}(\hat{I}_\theta)   \\
&= \delta^{-1}\mathcal{B}^2\sum_{\theta=1}^{m^3}  \sum_{G\in\mathcal{G}}  \alpha(G)^2  \sum_{k\in\mathbb{Z}}  \sum_{\substack{I\in\mathcal{I}_{k,1,2}\\ \sigma(F_{k,1,2}(I))\geq\delta_1 \sigma(I)\\  \pi_{\mathcal{G}} I_\theta =G}}
  \tilde{\mu}(\hat{I}_\theta),
\end{align*}
where the last inequality follows from the testing condition for $\mathsf{P}^{*}$.

We point out that for each dyadic cube $I$, the set
$$ \left\{k\in\mathbb{Z}:\ I\in\mathcal{I}_{k,1,2}, \sigma(F_{k,1,2}(I)) \geq \delta \sigma(I)\right\} $$
consists of at most $\delta_1^{-1}$ consecutive integers. Actually, that the integers
in this set are consecutive follows from the nested property of the collections $\mathcal{I}_k$.
Moreover, note that for each fixed $I$, the sets $F_{k,1,2}(I)\subset I$ are pairwise disjoint (with respect to $k$), and 
for each $k$,  $\sigma(F_{k,1,2}(I)) \geq \delta_1 \sigma(I)$. Hence, there are at most $\delta_1^{-1}$ such integers $k$.

As a consequence, we obtain that
\begin{align*}
B_{1,2,21}&\leq C\delta_1^{-2} \mathcal{B}^2\sum_{G\in\mathcal{G}}  \alpha(G)^2  
  \tilde{\mu}(\hat{G})\leq C\delta_1^{-2} \mathcal{B}^2  \|\phi\|_{L^2(M_{+};\mu)}^2,
\end{align*}
where the last inequality follows from the maximal inequality \eqref{maximal function}.  This gives \eqref{e:B1221estimate}.

We now turn to the estimate $B_{1,2,22}$,
\begin{align*}
&B_{1,2,22}\\
&\lesssim\delta_1^{-1}\sum_{\theta=1}^{m^3} \sum_{k\in\mathbb{Z}}  \sum_{\substack{I\in\mathcal{I}_{k,1,2}\\ \sigma(F_{k,1,2}(I))\geq\delta \sigma(I) }}
\frac{1}{\sigma(I)}  \left(\sum_{\substack{ J\in \mathcal{I}_{k+\ell_1+1,1,2}\\ J\subset I_\theta\\ \pi_\mathcal{G}J\subsetneq\pi_\mathcal{G} I_\theta}} \int_{\hat{J}}  \mathsf{P}_{\sigma,1,2}(1_{I})(x,t) \,\frac{d\tilde{\mu}(x,t)}{t}  \cdot \alpha(J) \right)^2 \\
&\lesssim\delta_1^{-1}\sum_{\theta=1}^{m^3} \sum_{k\in\mathbb{Z}}  \sum_{\substack{I\in\mathcal{I}_{k,1,2}\\ \sigma(F_{k,1,2}(I))\geq\delta \sigma(I) }}
\frac{1}{\sigma(I)}  \sum_{\substack{ J\in \mathcal{I}_{k+\ell_1+1,1,2}\\ J\subset I_\theta\\ \pi_\mathcal{G}J\subsetneq\pi_\mathcal{G} I_\theta}} \left[\int_{\hat{J}}  \mathsf{P}_{\sigma,1,2}(1_{I})(x,t)\, \frac{d\tilde{\mu}(x,t)}{t} \right]^{2}\tilde{\mu}(\hat{J})^{-1}  \\
&\quad\quad\times  \sum_{\substack{ J\in \mathcal{I}_{k+\ell_1+1,1,2}\\ J\subset I_\theta\\ \pi_\mathcal{G}J\subsetneq\pi_\mathcal{G} I_\theta}} \tilde{\mu}(\hat{J}) \alpha(J)^2,
\end{align*}
where the last inequality follows from Cauchy--Schwarz inequality. Next, from the Cauchy--Schwarz inequality, the definition of $\tilde{\mu}$ and the testing condition, we have
\begin{align*}
  \sum_{\substack{ J\in \mathcal{I}_{k+\ell_1+1,1,2}\\ J\subset I_\theta\\ \pi_\mathcal{G}J\subsetneq\pi_\mathcal{G} I_\theta}} \left[\int_{\hat{J}}  \mathsf{P}_{\sigma,1,2}(1_{I})(x,t) \,\frac{d\tilde{\mu}(x,t)}{t} \right]^{2}\tilde{\mu}(\hat{J})^{-1}   & \leq \sum_{\substack{ J\in \mathcal{I}_{k+\ell_1+1,1,2}\\ J\subset I_\theta\\ \pi_\mathcal{G}J\subsetneq\pi_\mathcal{G} I_\theta}} \int_{\hat{J}}  \mathsf{P}_\sigma(1_{I})(x,t)^2  \,d\mu(x,t)\\
& \leq \mathcal{F}^2 \sigma(I),
\end{align*}
which implies that
\begin{align}
B_{1,2,22}
&\lesssim\mathcal{F}^2\delta_1^{-1}\sum_{\theta=1}^{m^3} \sum_{k\in\mathbb{Z}}  \sum_{\substack{I\in\mathcal{I}_{k,1,2}\\ \sigma(F_{k,1,2}(I))\geq\delta \sigma(I) }}
\frac{1}{\sigma(I)}   \sigma(I)  \sum_{\substack{ J\in \mathcal{I}_{k+\ell_1+1,1,2}\\ J\subset I_\theta\\ \pi_\mathcal{G}J\subsetneq\pi_\mathcal{G} I_\theta}} \tilde{\mu}(\hat{J}) \alpha(J)^2\nonumber\\
&\lesssim\mathcal{F}^2\delta_1^{-1} \sum_{k\in\mathbb{Z}}  \sum_{\substack{I\in\mathcal{I}_{k,1,2}\\ \sigma(F_{k,1,2}(I))\geq\delta \sigma(I) }}
\sum_{\substack{ J\in \mathcal{I}_{k+\ell_1+1,1,2}\\ J\subset I_\theta, \pi_\mathcal{G}J\subsetneq\pi_\mathcal{G} I_\theta}} \tilde{\mu}(\hat{J}) \alpha(\pi_\mathcal{G}J)^2.\label{B1222}
\end{align}
As a similar result of Lemma \ref{L}, we get that 
\begin{align*}
B_{1,2,22}
&\lesssim \mathcal{F}^2\delta^{-1} \sum_{I\in \mathcal{G}} \tilde{\mu}(\hat{I}) \alpha(I)^2 \lesssim\mathcal{F}^2\delta^{-1}  \|\phi\|_{L^2(M_{+};\mu)}^2
\end{align*}
which is \eqref{e:B1222estimate}.

\subsubsection{Term $T_{1,3}$}
The estimate of $T_{1,3}$ is similar with the estimate of $T_{1,2}$, we leave details to the reader.

\subsubsection{Term $T_{1,4}$}
The estimates of  $T_{1,3}$ is similar with the estimate of $T_{1,2}$, if we have the following results.

\begin{lem}\label{lem P*14}
There exists a positive constant $C_{1,4}$ such that
\begin{align}
\mathsf{P}^{*}_{\mu,1,4}(\phi\cdot 1_{(3\hat{I})^c} )(x) < C_{1,4}2^k
\end{align}
for all $x\in I$, $I\in\mathcal{I}_{k,1,4}$ and $k\in\mathbb{Z}$.
\end{lem}
\begin{proof}
Note that $I$ is the Whitney cubes, satisfying $3I\subset \Omega_k$ and 
$9C_1\delta^{-1}I\not\subset\Omega_k$. We now choose $z\in 9C_1\delta^{-1}I\cap \Omega_k^c$. Then we obtain that
$\ell(I)<d(z,x)<5C_1\delta^{-1}\ell(I)$. Recall that $\mathsf{P}^{*}_{\mu,1,4}$ is the operator associated to the integral kernel
$ \mathsf{P}_{t,1,4}(x,y)$.
Since $x\in I$ and $3I\subset \Omega_k\subset \R^m\backslash K$, we have $d(x,K)\ge \ell(I)$.
For $z\in 9C_1\delta^{-1}I\cap \Omega_k^c$, it is clear that
$$|z|\approx 1+d(z,k)\lesssim 1+d(z,x)+d(x,K)\lesssim 1+d(x,K) \approx |x|.$$
Hence for $z\in 9C_1\delta^{-1}I\cap \Omega_k^c$ and for every $y$ with $(y,t)\not\in 3\hat{I}$,
there holds
\begin{equation*}
\begin{aligned}
\mathsf{P}_{t,1,4}(y,x)&=\frac{1}{t^n|y|^{m-2}|x|^{m-2}}\Big(\frac{t}{t+|y|+|x|}\Big)^{n+1} \\
&\leq C_{1,4}\frac{1}{t^n|y|^{m-2}|z|^{m-2}}\Big(\frac{t}{t+|y|+|z|}\Big)^{n+1}= C_{1,4}\mathsf{P}_{t,1,4}(y,z).
\end{aligned}
\end{equation*}
Now we multiply it by $\phi(y,t) 1_{(3\hat{I})^c}$ and then integrate with respect
to $d\mu(y,t)$. As a consequence, we have
\begin{align*}
\mathsf{P}^{*}_{\mu,1,2}(\phi\cdot 1_{(3\hat{I})^c} )(x)\leq C_{1,4} \mathsf{P}^{*}_{\mu,1,2}(\phi_1\cdot 1_{(3\hat{I})^c} )(z) \leq C_{1,4}2^k.
\end{align*}
The proof is complete.
\end{proof}

We have the result as similar with \eqref{claim P12}.
Since $J\in \mathcal{I}_{k+\ell_{1,4}+1} $, we have $3J\subset {\Omega}_{k+\ell_{1,4}+1,1,2}\subset \R^m \backslash K$ and hence $d(x,K)\ge \ell(J)$ for any $x\in J$.
Thus, for $(x,t)\in\hat{J}$, it follows $|x|\approx |x_J|$ and $t+|x|\approx \ell(J)+|x_J|$.
Moreover, we have that for $(x,t)\in\hat{J}$,
\begin{equation*}
\begin{aligned}
&\mathsf{P}_{\sigma,1,4}(1_{F_{k,1,4}(I)})(x,t) \\
&=\int_{\R^m \backslash K} \frac{1}{t^n|y|^{m-2}|x|^{m-2}}
\Big(\frac{t}{t+|y|+|x|}\Big)^{n+1} 1_{F_{k,1,4}(I)}(y)d\sigma (y)\\
&\approx  \frac{t}{\ell(J)}\int_{\R^m \backslash K} \frac{1}{\ell(J)^n|y|^{m-2}|x_J|^{m-2}}
\Big(\frac{\ell(J)}{\ell(J)+|y|+|x_J|}\Big)^{n+1}  1_{F_{k,1,4}(I)}(y)d\sigma (y)\\
&=\frac{t}{\ell(J)}\mathsf{P}_{\sigma,1,4}(1_{F_{k,1,4}(I)})(x_J,\ell(J)), 
\end{aligned}
\end{equation*}
where the implicit constants are independent of $x$, $t$ and $I$.

\subsection{Estimate for $T_{2}$}
By the definition of $P^*$ as in \eqref{P*}, 
we have that for $x\in \R^n\backslash K$,
\begin{align*}
 \mathsf{P}^{*}_\mu(\phi)(x)
 &=  \int_{M_{+}} \mathsf{P}_t(y,x)\phi(y,t) \,d\mu(y,t)\\
 &\le  \int_{M_{+}} \bigg(\frac{1}{t^m}\Big(\frac{t}{t+d(y,x)}\Big)^{m+ 1} +\frac{1}{t^n}\Big(\frac{t}{t+d(y,x)}\Big)^{n+ 1} \\
 &\qquad\qquad +\frac{1}{t^n|y|^{m-2}}\Big(\frac{t}{t+d(y,x)}\Big)^{n+ 1} + \frac{1}{t^m|x|^{n-2}}\Big(\frac{t}{t+d(y,x)}\Big)^{m+ 1}\bigg)\phi_1(y,t) \,d\mu(y,t)\\
 &=:\mathsf{P}^{*}_{\mu,2,1}(\phi_1)(x)+\mathsf{P}^{*}_{\mu,2,2}(\phi_1)(x)+\mathsf{P}^{*}_{\mu,2,3}(\phi)(x),
\end{align*}
where $\mathsf{P}^{*}_{\mu,2,1}$ is the operator associated to the integral kernel
$$ \mathsf{P}_{t,2,1}(y,x)=\frac{1}{t^m}\Big(\frac{t}{t+d(y,x)}\Big)^{m+ 1} ; $$ 
where $\mathsf{P}^{*}_{\mu,2,2}$ is the operator associated to the integral kernel
$$ \mathsf{P}_{t,2,2}(y,x)=\frac{1}{t^n}\Big(\frac{t}{t+d(y,x)}\Big)^{n+ 1} ; $$ 
 $\mathsf{P}^{*}_{\mu,2,3}$ is the operator associated to the integral kernel
$$ \mathsf{P}_{t,2,3}(y,x)=\frac{1}{t^n|y|^{m-2}}\Big(\frac{t}{t+d(y,x)}\Big)^{n+ 1}$$
and $\mathsf{P}^{*}_{\mu,2,4}$ is the operator associated to the integral kernel
$$ \mathsf{P}_{t,2,4}(y,x)=\frac{1}{t^m|x|^{n-2}}\Big(\frac{t}{t+d(y,x)}\Big)^{m+ 1}.$$
Then we have that
\begin{align*}
T_2&\ls \int_{\R^n\backslash K} | \mathsf{P}^{*}_{\mu,2,1}(\phi)(x)|^2 \,d\sigma(x)
+\int_{\R^n\backslash K} | \mathsf{P}^{*}_{\mu,2,2}(\phi)(x)|^2 \,d\sigma(x)\\
&\qquad +\int_{\R^n\backslash K} | \mathsf{P}^{*}_{\mu,2,3}(\phi)(x)|^2 \,d\sigma(x)+\int_{\R^n\backslash K} | \mathsf{P}^{*}_{\mu,2,4}(\phi)(x)|^2 \,d\sigma(x)\\
&=:T_{2,1}+T_{2,2}+T_{2,3}+T_{2,3}.
\end{align*}
The estimates of the term $T_{2,1}$ and $T_{2,2}$ are similar to 
the estimate of the term $T_{1,1}$, we omit the details.
We only show the term $T_{2,3}$ since $T_{2,3}$ and $T_{2,4}$ are similar.

\subsubsection{Term $T_{2,3}$}\label{T22}

Set
\begin{align*}
\Omega_{k,2,3}&:=\left\{ x\in \R^n\backslash K: \mathsf{P}^{*}_{\mu,2,2}(\phi)(x)  >2^k \right\}.
\end{align*}

Let $\ell_2$ be a large constant to determined later. Then we apply the Whitney decomposition
to $\Omega_{k,2,3}$  and denote by $\mathcal{I}_{k,2,3}$ the dyadic cubes. 
Then we have
\begin{align*}
\int_{\R^n\backslash K} | \mathsf{P}^{*}_{\mu,2,3}(\phi)(x)|^2 \,d\sigma(x) 
&= \sum_{k\in\mathbb{Z}}\int_{\Omega_{k+\ell_2,2,3}\backslash\Omega_{k+\ell_2+1,2,3}} | \mathsf{P}^{*}_{\mu,2,3}(\phi)(x)|^2 \,d\sigma(x) \\
&\lesssim  \sum_{k\in\mathbb{Z}} 2^{2k}\sigma(\Omega_{k+\ell_2,2,3}\backslash\Omega_{k+\ell_2+1,2,3})\\
&=\sum_{k\in\mathbb{Z}} 2^{2k} \sum_{I\in\mathcal{I}_{k,2,3
}}\sigma(I\cap (\Omega_{k+\ell_2,2,3}\backslash\Omega_{k+\ell_2+1,2,3})).
\end{align*}
Set $F_{k,2,3}(I):= I \cap (\Omega_{k+\ell_2,2,3}\backslash\Omega_{k+\ell_2+1,2,3})$ with 
$I\in \mathcal{I}_{k,2,3}$. Now let $\delta_2\in(0,1)$, to be chosen sufficiently small.  Then we have
\begin{align*}
&\int_{\R^n\backslash K} | \mathsf{P}^{*}_{\mu,2,3}(\phi)(x)|^2 \,d\sigma(x) \\
&=\sum_{k\in\mathbb{Z}} 2^{2k} \sum_{\substack{I\in\mathcal{I}_{k,2,3}\\ \sigma(F_{k,2,3}(I))<\delta_2 \sigma(I) }}\sigma(F_{k,2,3}(I))
+\sum_{k\in\mathbb{Z}} 2^{2k} \sum_{\substack{I\in\mathcal{I}_{k,2,3}\\ \sigma(F_{k,2,3}(I))\geq\delta_2 \sigma(I) }}\sigma(F_{k,2,3}(I))\\
&=:A_{2,3}+B_{2,3}.
\end{align*}
As for the term $A_{2,3}$, it is the same to $A_{1,1}$. Thus it remains to show that term $B_{2,3}$ can be dominated in terms of the testing conditions. Since
$$ \mathsf{P}_{t,2,3}(y,x)=\frac{1}{t^n|y|^{m-2}}\Big(\frac{t}{t+d(y,x)}\Big)^{n+ 1} , $$ we use the same proof of Lemma \ref{lem P*} to get the following lemma.
\begin{lem}\label{lem P*22}
There exists a positive constant $C_2$ such that
\begin{align}
\mathsf{P}^{*}_{\mu,2,3}(\phi\cdot 1_{(3\hat{I})^c} )(x) < C_22^k
\end{align}
for all $x\in I$, $I\in\mathcal{I}_{k,2,3}$ and $k\in\mathbb{Z}$.
\end{lem}

Now for $I\in\mathcal{I}_{k,2,3}$ with $\sigma(F_{k,2,3}(I))\geq\delta_2 \sigma(I) $ and for each $x\in F_{k,2,3}(I)$, it follows from
the above lemma that
\begin{align*}
\mathsf{P}^{*}_{\mu,2,3}(\phi\cdot 1_{3\hat{I}} )(x) &=
\mathsf{P}^{*}_{\mu,2,3}(\phi )(x) -  \mathsf{P}^{*}_{\mu,2,3}(\phi\cdot 1_{(3\hat{I})^c} )(x) 
\geq 2^{k+\ell_2} - C_22^k,
\end{align*}
where $C_2$ is the constant from Lemma \ref{lem P*22}.
By choosing $\ell_2$ such that $2^{\ell_2}>C_2+1$, we obtain that 
\begin{align*}
\mathsf{P}^{*}_{\mu,2,3}(\phi\cdot 1_{3\hat{I}} )(x) \geq 2^k. 
\end{align*}
Hence,
\begin{align*}
2^k &\leq \frac{1}{\sigma(F_{k,2,3}(I))} \int _{F_{k,2,3}(I)} \mathsf{P}^{*}_{\mu,2,3}(\phi\cdot 1_{3\hat{I}} )(x)\,d\sigma(x) \\
&=\frac{1}{\sigma(F_{k,2,3}(I))}  \int_{3\hat{I}\backslash \hat{\Omega}_{k+\ell_2+1,2,3} } \mathsf{P}_{\sigma,2,3}( 1_{F_{k,2,3}(I)} )(x,t) \phi(x,t) \,d\mu(x,t)\\
&\quad+\frac{1}{\sigma(F_{k,2,3}(I))}  \int_{3\hat{I} \cap  \hat{\Omega}_{k+\ell_2+1,2,3}} \mathsf{P}_{\sigma,2,3}( 1_{F_{k,2,3}(I)} )(x,t) \phi(x,t) \,d\mu(x,t)\\
&=: B_{2,3,1}(k,I)+B_{2,3,2}(k,I),
\end{align*}
where $\mathsf{P}_{\sigma,2,3}$ is the dual operator of $\mathsf{P}^{*}_{\mu,2,3}$.  This gives that
\begin{align*}
B_{2,3}
&\leq 2\sum_{k\in\mathbb{Z}}  \sum_{\substack{I\in\mathcal{I}_{k,2,3}\\ \sigma(F_{k,2,3}(I))\geq\delta_2 \sigma(I) }}
B_{2,3,1}(k,I)^2\sigma(F_{k,2,3}(I)) \\
&\qquad +2\sum_{k\in\mathbb{Z}}  \sum_{\substack{I\in\mathcal{I}_{k,2,3}\\ \sigma(F_{k,2,3}(I))\geq\delta_2 \sigma(I) }}
B_{2,2,3}(k,I)^2\sigma(F_{k,2,3}(I))\\
&=:B_{2,3,1}+B_{2,3,2}.
\end{align*}
The estimate of the term $B_{2,3,1}$ is similar to 
the  estimate of the term $B_{1,1,1}$, we have
$$B_{2,3,1}  \lesssim  \delta_2^{-2}\mathcal{F}^2\left\Vert \phi\right\Vert_{L^2(M_{+};\mu)}^2.$$
We only show that the term $B_{2,3,2}$ satisfies the following inequality
\begin{equation}\label{e:B222estimate}
B_{2,3,2} \lesssim  \delta_2^{-2}\left(\mathcal{F}^2+\mathcal{B}^2\right)\left\Vert \phi\right\Vert_{L^2(M_{+};\mu)}^2
\end{equation}
and then
\begin{equation*}
B_{2,3}\lesssim \left(\mathcal{F}^2+\mathcal{B}^2\right)\left\Vert \phi\right\Vert_{L^2(M_{+};\mu)}^2.
\end{equation*}

We now estimate $B_{2,3,2}$, which is bounded by
\begin{align}\label{B222}
B_{2,3,2}&\lesssim\delta_2^{-1}\sum_{k\in\mathbb{Z}}  \sum_{\substack{I\in\mathcal{I}_{k,2,3}\\ \sigma(F_{k,2,3}(I))\geq\delta_2 \sigma(I) }}
\frac{1}{\sigma(I)}  \left(\int_{3\hat{I} \cap  \hat{\Omega}_{k+\ell_2+1,2,3}} \mathsf{P}_{\sigma,2,3}( 1_{F_{k,2,3}(I)} )(x,t) \phi(x,t) \,d\mu(x,t) \right)^2. 
\end{align}
To continue, we decompose 
\begin{align}\label{B222 decom}
 3\hat{I}\cap  \hat{\Omega}_{k+\ell_2+1,2,3}=\bigcup_{J} \{ \hat{J}: J\subset 3I, J\in \mathcal{I}_{k+\ell_2+1,2,3}  \}.
\end{align}
Note that for such $J$, $3J\cap F_{k,2,3}(I)=\emptyset$. For $(x,t)\in \hat J$ and $y\in F_{k,2,3}(I)$, we have $t+d(x,y)\approx \ell(J)+d(x_J,y)$. 
Moreover, we have that for $(x,t)\in\hat{J}$,
\begin{equation}\label{claim P23}
\begin{aligned}
&\mathsf{P}_{\sigma,2,3}(1_{F_{k,2,3}(I)})(x,t) \\
&=\int_{\R^n \backslash K} \frac{1}{t^n|y|^{m-2}}\Big(\frac{t}{t+d(y,x)}\Big)^{n+ 1} 1_{F_{k,2,3}(I)}(y)d\sigma y\\
&\approx  \frac{t}{\ell(J)}\int_{\R^n \backslash K}\frac{1}{\ell(J)^n|y|^{m-2}}\Big(\frac{\ell(J)}{\ell(J)+d(y,{\color{red}x_J})}\Big)^{n+ 1} 1_{F_{k,2,3}(I)}(y)d\sigma y\\
&=\frac{t}{\ell(J)}\mathsf{P}_{\sigma,2,2}(1_{F_{k,2,3}(I)})(x_J,\ell(J)), 
\end{aligned}
\end{equation}where the implicit constants are independent of $x$, $t$ and $I$.

From \eqref{claim P23},  we use the approach used to prove \eqref{B12 e2} to obtain
\begin{equation}\label{B22 e2}
\begin{aligned}
&\int_{\hat{J}}\mathsf{P}_{\sigma,2,3}(1_{F_{k,2,3}(I)})(x,t)\phi(x,t)\,d\mu(x,t)\\
&\qquad\lesssim \int_{\hat{J}}  \mathsf{P}_{\sigma,2,3}(1_{I})(x,t) \frac{1}{t} d\tilde{\mu}(x,t) \cdot \frac{1}{\tilde{\mu}(\hat{J})} \cdot\int_{\hat{J}} \frac{1}{t} \phi(x,t)\,d\tilde{\mu}(x,t), 
\end{aligned}
\end{equation}
where $d\tilde{\mu}(x,t) =t^2\,d\mu(x,t)$.  By \eqref{B222}, \eqref{B222 decom} and
\eqref{B22 e2}, we obtain that
$$
B_{2,3,2}
\lesssim \delta_2^{-1}\sum_{k\in\mathbb{Z}}  \sum_{\substack{I\in\mathcal{I}_{k,2,3}\\ \sigma(F_{k,2,3}(I))\geq\delta_2 \sigma(I) }}\frac{1}{\sigma(I)}
 \left(\sum_{\substack{ J\in \mathcal{I}_{k+\ell_2+1,2,3}\\ J\subset 3I }} \int_{\hat{J}}  \mathsf{P}_{\sigma,2,3}(1_{I})(x,t) \,\frac{d\tilde{\mu}(x,t)}{t}  \cdot \alpha(J)  \right)^2. 
$$
As a similar method of $B_{1,2,2}$, we have
$$B_{2,3,2} \lesssim  \delta_2^{-2}\left(\mathcal{F}^2+\mathcal{B}^2\right)\left\Vert \phi\right\Vert_{L^2(M_{+};\mu)}^2,$$
which is \eqref{e:B222estimate}.

\subsection{Estimate for $T_{3}$}
By the definition of $P^*$ as in \eqref{P*}, 
we have that for $x\in K$,
\begin{align*}
 \mathsf{P}^{*}_\mu(\phi)(x) &=  \int_{M_{+}} \mathsf{P}_t(y,x)\phi(y,t) \,d\mu(y,t)\\
 &\le  \int_{M_{+}} \bigg(\frac{1}{t^m}\Big(\frac{t}{t+d(y,x)}\Big)^{m+ 1} +\frac{1}{t^n}\Big(\frac{t}{t+d(y,x)}\Big)^{n+ 1}\\
&\qquad +\frac{1}{t^n|y|^{m-2}}\Big(\frac{t}{t+d(y,x)}\Big)^{n+ 1} \bigg)\phi(y,t) \,d\mu(y,t)\\
 &=:\mathsf{P}^{*}_{\mu,3,1}(\phi)(x)+\mathsf{P}^{*}_{\mu,3,2}(\phi)(x)+\mathsf{P}^{*}_{\mu,3,3}(\phi)(x),
\end{align*}
where $\mathsf{P}^{*}_{\mu,3,1}$ is the operator associated to the integral kernel
$$ \mathsf{P}_{t,3,1}(y,x)=\frac{1}{t^m}\Big(\frac{t}{t+d(x,y)}\Big)^{m+ 1}, $$ 
$\mathsf{P}^{*}_{\mu,3,2}$ is the operator associated to the integral kernel
$$ \mathsf{P}_{t,3,2}(y,x)=\frac{1}{t^n}\Big(\frac{t}{t+d(x,y)}\Big)^{n+ 1}  $$ 
and $\mathsf{P}^{*}_{\mu,3,3}$ is the operator associated to the integral kernel
$$ \mathsf{P}_{t,3,3}(y,x)=\frac{1}{t^n|y|^{m-2}}\Big(\frac{t}{t+d(x,y)}\Big)^{n+ 1} . $$

Then we have that
\begin{align*}
T_3&\ls \int_{K} | \mathsf{P}^{*}_{\mu,3,1}(\phi)(x)|^2 \,d\sigma(x)
+\int_{K} | \mathsf{P}^{*}_{\mu,3,2}(\phi)(x)|^2 \,d\sigma(x)+\int_{K} | \mathsf{P}^{*}_{\mu,3,2}(\phi)(x)|^2 \,d\sigma(x)\\
&=:T_{3,1}+T_{3,2}+T_{3,2}.
\end{align*}
The estimates of $T_{3,1}$ and $T_{3,2}$ are similar with the estimates of $T_{1,1}$
and the estimates of $T_{3,3}$ is  similar with the estimate $T_{2,2}$.  We leave the precise details to the reader.
\bigskip
\bigskip

\noindent{\bf Acknowledgement:} The authors would like to express their sincere gratitude to the referees for their careful reading, patient reviewing, valuable corrections and constructive comments, which greatly improved the exposition of this manuscript. The authors would like to thank Chun-Yen Shen for helpful discussions on arXiv version in 2021 of this paper.

\end{document}